\documentclass{amsart}
\usepackage{amssymb,mathrsfs}

\newcommand{\Span}{\operatorname{span}}
\newcommand{\Ker}{\operatorname{Ker}}
\newcommand{\sign}{\operatorname{sign}}
\newcommand{\rad}{_{\operatorname{radial}}}

\newcommand{\sS}{\mathscr{S}}
\newcommand{\sL}{\mathscr{L}}
\newcommand{\sV}{\mathscr{V}}
\newcommand{\sK}{\mathscr{K}}

\newcommand{\cH}{\mathcal{H}}
\newcommand{\cI}{\mathcal{I}}
\newcommand{\cS}{\mathcal{S}}
\newcommand{\cW}{\mathcal{W}}
\newcommand{\cB}{\mathcal{B}}
\newcommand{\cL}{\mathcal{L}}
\newcommand{\cX}{\mathcal{X}}
\newcommand{\R}{\mathbb{R}}
\newcommand{\C}{\mathbb{C}}
\newcommand{\Z}{\mathbb{Z}}
\newcommand{\na}{\nabla}
\newcommand{\ti}{\tilde}
\newcommand{\ck}{\check}

\newcommand{\I}{\infty}
\newcommand{\p}{\partial}
\newcommand{\al}{\alpha}
\newcommand{\be}{\beta}
\newcommand{\ga}{\gamma}
\newcommand{\de}{\delta}
\newcommand{\ep}{\epsilon}
\newcommand{\si}{\sigma}
\newcommand{\te}{\theta}
\newcommand{\ta}{\tau}
\newcommand{\la}{\lambda}
\newcommand{\ka}{\kappa}
\newcommand{\fy}{\varphi}
\newcommand{\y}{\eta}
\newcommand{\om}{\omega}
\newcommand{\De}{\Delta}
\newcommand{\Te}{\Theta}

\newcommand{\sg}{\mathfrak{s}}
\newcommand{\gec}{\,\gtrsim\,}
\newcommand{\lec}{\,\lesssim\,}
\newcommand{\LR}[1]{{\langle #1 \rangle}}
\newcommand{\EQ}[1]{\begin{equation}\begin{split} #1 \end{split}\end{equation}}
\newcommand{\CAS}[1]{\begin{cases} #1 \end{cases}}
\newcommand{\MAT}[1]{\begin{pmatrix} #1 \end{pmatrix}}
\newcommand{\pt}{&} \newcommand{\pr}{\\ &} \newcommand{\pq}{\quad}
\newcommand{\pn}{} \newcommand{\prq}{\\ &\quad} \newcommand{\prQ}{\\ &\qquad}

\newcommand\supp{\operatorname{supp}}

\theoremstyle{plain}
\newtheorem{thm}{Theorem}

\newtheorem{rem}[thm]{Remark}
\newtheorem{claim}[thm]{Claim}
\newtheorem{prop}[thm]{Proposition}
\newtheorem{lem}[thm]{Lemma}

\numberwithin{equation}{section} \numberwithin{thm}{section}

\begin{document}

\title[Global dynamics for energy-critical Schr\"odinger equation]
{Global dynamics above the ground state \\ for the energy-critical Schr\"odinger equation \\ with radial data}

\author{Kenji Nakanishi}
\address{Department of Pure and Applied Mathematics
Graduate School of Information Science and Technology
Osaka University, Toyonaka, Osaka 560-0043, JAPAN}
\email{nakanishi@ist.osaka-u.ac.jp}

\author{Tristan Roy}
\address{Graduate School of Mathematics, Nagoya University}
\email{tristanroy@math.nagoya-u.ac.jp}


\begin{abstract}
Consider the focusing energy critical Schr\"odinger equation in three space dimensions
with radial initial data in the energy space. We describe the global dynamics of all the solutions of which the energy is at most slightly larger than that of the ground states, according to whether it stays in a neighborhood of them, blows up in finite time or scatters. 
In analogy with \cite{nakaschlagschrod}, the proof uses an analysis of the hyperbolic dynamics near them and the variational structure far from them. The key step that allows
to classify the solutions is the \textit{one-pass} lemma. The main difference between
\cite{nakaschlagschrod} and this paper is that one has to introduce a scaling parameter in order to describe
the dynamics near them. One has to take into account this parameter in the analysis around the ground states by
introducing some orthogonality conditions. One also has to take it into account in the proof of the \textit{one-pass} lemma by comparing the contribution in the variational region and in the hyperbolic region.
\end{abstract}

\maketitle

\tableofcontents

\section{Introduction}
In this paper, we consider the semilinear Schr\"odinger equation on $\mathbb{R}^3$ with the focusing energy-critical power for $u=u(t,x):\R^{1+3}\to\C$:
\EQ{ \label{Eqn:SchrodCrit}
 \pt i \partial_{t} u  - \De u   =   |u|^4 u,
 \pq u(0,x)=u_0(x)}
with radial initial data $u_{0} \in \dot{H}^{1}$ (or $H^{1}$).
Here $\dot{H}^{1}$ (resp.~$H^{1}$) is the standard homogeneous (resp.~inhomogeneous) Sobolev space in three dimensions, i.e., the completion of the Schwartz space with respect to the norm $ \| f \|_{\dot{H}^{1}} := \| \nabla f \|_{L^{2}} $
(resp.~$\| f \|_{H^{1}} := \| f \|_{L^{2}} + \| \nabla f \|_{L^{2}} $).
Our consideration is restricted throughout this paper to the radial subspace:
\EQ{
 \dot H^1\rad:=\{\fy\in\dot H^1 \mid \fy(x)=\fy(|x|)\}.}
A strong solution of \eqref{Eqn:SchrodCrit} is a solution that satisfies the Duhamel formula:
\EQ{ \label{Eqn:StrongSol}
 u(t) & = e^{- i t \De} u_{0} - i \int_{0}^{t} e^{-i(t-t')\De} ( |u|^4 u(t')) \,  dt'.}
It enjoys the following energy conservation law
\EQ{ \label{Eqn:Nrj}
 E(u(t)) & := \frac{1}{2} \int_{\R^3} |\nabla u(t,x)|^{2} \, dx
- \frac{1}{6} \int_{\R^3} |u(t,x)|^{6} \, dx = E(u(0)).}
\eqref{Eqn:SchrodCrit} can be written in the Hamiltonian form
$\partial_{t} u   = i E'(u)$,
where
$\LR{E'(u),h}= \partial_{\lambda} E(u + \lambda h)|_{\lambda =0}
 = -  \LR{\De u, h} - \LR{|u|^4 u, h}$,
and $\LR{\cdot,\cdot}$ denotes the real-valued inner product on $L^2(\R^3)$:
\EQ{
 (f | g):= \int_{\R^3} f(x)\bar{g}(x) dx, \pq \LR{f,g} := \Re(f | g). }
The symplectic form  $\omega$ associated to this Hamiltonian system is
\EQ{
 \omega (u, v) := \LR{i u,v}. }
This equation admits a family of radial stationary solutions called the ground states, described by the rotation parameter $\theta\in\R$ and the scaling parameter $\sigma\in\R$:
\EQ{
 W_{\theta,\sigma}(x):= e^{i \theta} e^{\sigma/2}  W (e^{\sigma} x)\in\dot H^1\rad,}
with
\EQ{ \label{Eqn:ExplFormW}
 W(x) := \left( 1 + \frac{|x|^{2}}{3} \right)^{-1/2}, }
which satisfy
\EQ{ \label{Eqn:StatSolW}
 - \De W_{\theta, \sigma} & = |W_{\theta,\sigma}|^4 W_{\theta,\sigma}. }
The two dimensional manifold of those stationary solutions is denoted by
\EQ{
 \cW := \{ W_{\te,\si} \mid \te,\si\in\R\}\subset\dot H^1\rad.}
The distance from $\cW$ and its $\de$-neighborhood are denoted by
\EQ{ \label{def dW}
 d_\cW(\fy):=\inf_{\te,\si\in\R}\|\fy-W_{\te,\si}\|_{\dot H^1},
 \pq B_\de(\cW):=\{\fy\in\dot H^1\rad \mid d_\cW(\fy)<\de\}.}
Note that this set is invariant for the complex rotation and $\dot H^1$ scaling.
Recall (see \cite{aubin,tal}) that $W$ is an extremizer for the Sobolev inequality, i.e.,
\EQ{ \label{Eqn:SobIneqMax}
 \| W \|_{L^6}/ \| W \|_{\dot{H}^{1}}= \sup \{ \|f\|_{L^6}/\|f\|_{\dot H^1} \mid 0\not=f  \in \dot{H}^{1} \}. }
The local well-posedness of \eqref{Eqn:SchrodCrit} has been studied in \cite{cazpaper,cazbook}.
See \cite{kenmer} for a summary of these results.
In particular, it is known that on an interval $J$ such that $\| e^{i t \triangle} u_{0} \|_{L_{t,x}^{10}(J\times\R^3)}$ is small enough, there exists a unique solution $u$ of \eqref{Eqn:StrongSol} in a subspace of $C(J,\dot H^{1})$.
This allows us to define the maximal time interval of existence $I(u):=(-T_{-}(u), T_{+}(u) )$ with $T_{+}(u)$, $T_{-}(u)$ denoting respectively the forward, backward maximal time of existence (in this class): see again \cite{kenmer} for more detail.
The next step is to understand the global behavior of \eqref{Eqn:StrongSol}.
Classification of radial solutions of \eqref{Eqn:SchrodCrit} was studied for $E(u_0)<E(W)$ in \cite{kenmer}, and that for $E(u_0)=E(W)$ in \cite{duymerle}.
These results are summarized as follows: For $u_0\in\dot H^1\rad$ with $E(u_0)\le E(W)$,
\begin{itemize}
\item If $\| \nabla u_{0} \|_{L^{2}} < \| \nabla  W \|_{L^{2}} $, 
then the solution is either $W^-$ up to symmetry, or scattering as $t\to\pm\I$, i.e., $T_\pm(u)=\I$ and there exist $u_\pm \in \dot{H}^{1}$ such that $ \lim\limits_{t \rightarrow \pm \infty} \| u(t) - e^{-i t \De} u_{\pm} \|_{\dot{H}^{1}} =0$.
\item If $\|\na u_0\|_{L^2}=\|\na W\|_{L^2}$ then $u(t)=u(0)\in\cW$. 
\item If $\| \nabla u_{0} \|_{L^{2}} > \| \nabla W \|_{L^{2}} $ and $u_0 \in L^{2}(\R^{3})$ 
then the solution is either $W^+$ up to symmetry, or blowing up both in $t>0$ and in $t<0$ (i.e., $T_\pm(u)<\I$),
\end{itemize}
where $W^\pm$ are the unique solutions which converge to $W$ strongly in $\dot H^1$ as $t\to\I$, satisfying $\pm(\|\na W^\pm\|_{L^2}-\|\na W\|_{L^2})>0$.
$W^-$ scatters as $t\to-\I$, while $W^+$ blows up in $t<0$.

The goal of this paper is to classify the global behavior of solutions with slightly more energy than the ground states.
Our main result is the following. Let
\EQ{
 \pt \cH^\ep  := \{ \fy \in \dot H^1\rad \mid E(\fy) < E(W) + \epsilon^{2}\},
 \pr \cS_\pm:=\{u_0 \in \dot H^1\rad \mid \text{the solution $u$ scatters as $t\to\pm\I$}\},
 \pr \cB_\pm:=\{u_0 \in \dot H^1\rad \mid \text{the solution $u$ blows up in $\pm t>0$} \}.}

\begin{thm} \label{Thm:Main}
There is an absolute constant $\ep_\star\in(0,1)$ such that for each $\ep\in(0,\ep_\star]$, there exist a relatively closed set $\cX_\ep\subset\cH^\ep$, and a continuous function $\Te:\cH^{\ep}\setminus\cX_\ep\to\{\pm 1\}$, with the following properties.
$\cW\subset\cX_\ep\subset B_{C\ep}(\cW)$ for some absolute constant $C\in(0,\I)$. The values of $\Te$ are independent of $\ep$.
For each $u_0\in\cH^{\ep}$ and the solution $u$ of \eqref{Eqn:SchrodCrit},
\EQ{
 I_0(u)  := \{ t \in I(u) \mid u(t)\in\cX_\ep\}}
is either empty or an interval. Hence $I(u)\setminus I_0(u)$ consists of at most two open intervals.
Let $\si\in\{\pm\}$. If $\Te(u(t))=+1$ for $t$ close to $T_\si(u)$, then $u_0\in\cS_\si$.
If $\Te(u(t))=-1$ for $t$ close to $T_\si(u)$ and $u_0\in L^2(\R^3)$, then $u_0\in\cB_\si$.
\end{thm}
In other words, every solution with energy less than $E(W)+\ep^2$ can stay in $\cX_\ep$  only for an interval of time, though it can be the entire existence time.
Once the solution gets out of $\cX_\ep$ , it has to either scatter or blow-up, according to the sign function $\Te(u)$, though we need an additional condition $u_0\in L^2(\R^3)$ to ensure the blow-up.

The above properties hold in both the time directions. Concerning the relation between forward and backward dynamics, we have
\begin{thm} \label{Thm:4 sets}
For any $\ep>0$, each of the 4 intersections
\EQ{
 \cS_-\cap\cS_+,\pq \cB_-\cap\cS_+,\pq \cS_-\cap\cB_+,\pq \cB_-\cap\cB_+}
has non-empty interior in $\cH^\ep\cap L^2(\R^3)$.
\end{thm}
In particular, there are infinitely many solutions which scatter on one side of time and blow up on the other.
According to the previous theorem, such transition can occur only by changing $\Te(u)$ from $+1$ to $-1$ or vice versa, going through $O(\ep)$ neighborhood of $\cW$, but the change is allowed at most once for each solution.
Note however that there may well exist blow-up inside the neighborhood of $\cW$, as the equation is energy-critical. It is indeed the case for the energy-critical wave equation.
More precise dynamics around $\cW$ should be studied elsewhere.

In the proof of the above results, we will explicitly construct, in terms of the eigenfunctions of the linearized operator, the functionals $\ti d_\cW$ and $\Te$, as well as open initial data sets in the 4 intersections.

Now we explain the main ideas of this paper and how it is organized. The proof of Theorem \ref{Thm:Main} relies upon a strategy that was pioneered by the first author and Schlag in \cite{nakschlagkg} in the study of the nonlinear Klein-Gordon equation with the focusing cubic nonlinearity. It relies upon two components: an \textit{ejection lemma} and a \textit{one-pass lemma}.

\medskip

The \textit{ejection lemma} aims at describing the dynamics of the solution when it is in the exit mode, i.e., when it is close to $\cW$ and moving away from it. By analogy with the dynamics of solutions of linear differential equations, we would like its dynamics to be ruled by that of its unstable eigenmode of the linearized operator around $\cW$.
In order to verify this statement, one has to control the orthogonal component of the spectral decomposition of the remainder resulting from the linearization  around $\cW$.
We would like to control this component by using the quadratic terms resulting from the Taylor expansion of the energy around $\cW$.
This can be done if and only if the remainder satisfies two orthogonality conditions: see Proposition \ref{Prop:ControlOrth}.
In order to satisfy these conditions, one has to give two degrees of freedom to the decomposition of the solution around $\cW$: a rotation parameter (this was done in \cite{nakaschlagschrod}) and a scaling parameter: see Propositions \ref{Prop:OrthDecompu} and \ref{Prop:LinearParamW}.
Then, one also has to control the evolution of these two parameters.
We prove in Proposition \ref{Prop:DynEjection} that we can close the argument.
More precisely the dynamic of the solution close to $\cW$ and in the exit mode is dominated by the exponential growth of the unstable eigenmode;
moreover, a relevant functional (denoted by $K$) grows exponentially and its sign eventually becomes opposite to that of the eigenmode.

\medskip

The \textit{one-pass lemma} (see Proposition \ref{Prop:OnePassLemma} and Section \ref{Section:OnePassLemma} for more details) aims at classifying the fate of the solution.
A direct consequence of  this lemma is the dichotomy described in the statement of Theorem \ref{Thm:Main}.
It shows that the orbit cannot cross a neighborhood of $\cW$ more than once.
The proof is by contradiction.
Assuming that the solution crosses this neighborhood more than once, then it means that the solution is at two different times $t_{a}$, $t_{b}$ close to $\cW$ and in the ejection mode (forward and backward respectively in time).
So we can apply the ejection lemma as long as we are not so far from $\cW$ and then variational estimates (see Proposition \ref{Prop:VarEst}) far from $\cW$.
The contradiction appears when we integrate by part a localized virial identity \eqref{Eqn:RHS}.
The left-hand side is much smaller than the right-hand side thanks to the exponential growth of a relevant functional (denoted by $K$) in the ejection mode and variational estimates far from $\cW$.
The process involves a parameter $m$, which is the cut-off radius for the localization.
Notice that unlike the subcritical case, one has to take into account the scaling parameter defined by the ground states to which the solution is close to.
This requires a much more complicated analysis since we have no control of this parameter.
It is also harder than the energy-critical wave equation, for which it is easy to localize virial and energy estimates in space-time, thanks to the finite speed of propagation (see \cite{KNS}).
Indeed, this part of analysis is the main novelty of this paper.
In the case where $\Te(u(t))=+1$ after ejection (see Section \ref{Section:OnePassLemma}), one introduces a radius of the concentration of the kinetic part of the energy (see definition of $m_{V}^+$) and the hyperbolic parameter (see definition of $m_{H}$), estimates $K$ along with some error terms (generated by the cut-off) in the hyperbolic region and the variational region, and compares these estimates.
In the worst scenario, one proves a decay estimate (see \eqref{u/r small}) in the variational region and uses this estimate to implement Bourgain's energy induction method \cite{bourgjams}: this allows to construct a solution whose energy is smaller than the original one by a nontrivial amount (in particular it is smaller than that of the ground states), then the theory below the ground state energy (see \cite{kenmer}) implies that it is not close to the ground states, neither is the original solution by a perturbation argument, contradicting the assumption of returning orbit.
In the case where $\Te(u(t))=-1$ after ejection, we introduce a threshold (see definition of $m_{V}^-$) that allows to compare $K$ with the main part of the virial identity; then, by integrating the virial identity, we can prove that this threshold must be very large; then, by proving a decay estimate, one can show that this threshold is not so large, which leads to a contradiction.

\medskip

The fate of the solution depends on $\Te(u(t))$ when $u$ is ejected. If $\Te(u(t))=-1$ and $u_0\in L^2$ then we prove that it blows up in finite time; if $\Te(u(t))=+1$ then we prove that it is scattering: see Section \ref{Section:LongTimeFarGd}.
The scattering is proved by a modification of Kenig-Merle approach \cite{kenmer} and arguments from \cite{nakaschlagschrod}.
Unlike the subcritical case, one has to deal with possible blow-up in finite time, although the $\dot{H}^{1}$ norm is bounded.
The proof is by contradiction.
Assuming that scattering fails, then one can find a critical level of energy above which scattering does not hold for solutions that are far from the ground states and $\Te(u)=+1$.
But this means that there exists a sequence $(u_{n})_{n \geq 1}$ that satisfies the properties that we have just mentioned (in fact, the distance can be upgraded from far to very far, by appealing to the ejection lemma), and, thanks to a concentration compactness procedure, the fact that the energy of $u_{n}$ is just above that of the ground states, one can construct a critical element $U_{c}$ that does not scatter, has energy equal to the critical level of energy, is far from the ground states, and satisfies $\Te(U_{c})=+1$.
Moreover its orbit is precompact up to scaling.
By using Kenig-Merle's arguments, one sees that $U_{c}$ does not exist.

$\textbf{Acknowledgments}:$ The second author would like to thank W. Schlag and T. Duyckaerts for interesting
discussions related to this problem while he visited University of Chicago in April 2012 and IHP in June 2012. The
second author was supported by a JSPS fellowship.

\section{Notation}
In this section, we set up some notation that appear in this paper.
If $x$ is a complex number then $x = \Re{(x)} + i \Im { (x) } = x_{1} + i x_{2}$.
Here $\Re{(x)}$ and $x_{1}$ (resp.~$\Im{(x)}$
and $x_{2}$) denote the real part (resp.~the imaginary part) of $x$.
Given $x$, $y$ two real numbers, $x \lesssim y$ (resp.~$\gtrsim$) means that there exists a universal constant $C > 0 $ such that $x \leq C y$ (resp.~$x \geq C y $).
For any function $f$ on $\R$ or $[0,\I)$, and for any $m \in (0,\I)$, we denote by $f_{m}$ the following rescaled function
\EQ{
 f_{m}(r) := f(r/m).}

$L^p_x=L^p$ denotes the standard $L^p$ space on $\R^3$.
Some estimates that we establish in this paper require the Littlewood-Paley technology, which we set up now.
The Fourier transform of $\fy\in\cS'(\R^3)$ is denoted by $\widehat\fy$.
Let $\phi:\R\to[0,\I)$ be a smooth even function satisfying $t\phi'(t) \le 0$ and
\EQ{ \label{def phi}
 |t|\le 1\implies \phi(t)=1, \pq |t|\ge 2 \implies \phi(t)=0.}
The complement of this smooth cut-off is denoted by
\EQ{
 \phi^C:=1-\phi. }
For any $m>0$, Littlewood-Paley operators $P_{<m}$, $P_{\ge m}$ and $P_m$ are defined by
\EQ{
 \pt \widehat{P_{< m} f}(\xi) := \phi_m(|\xi|) \widehat{f}(\xi),
 \pq P_{\geq m}:= 1 - P_{< m},
 \pq P_{m}:=P_{<m}-P_{<m/2}.}

The following functionals on $\dot H^1(\R^3)$ play crucial roles in variational arguments.
\begin{align}
 & K(f)  := \| \nabla f \|^{2}_{L^{2}} - \| f \|^{6}_{L^{6}},
 \\ \label{Eqn:DefI}
 & I(f) := E(f) - K(f)/2 = \| f \|^{6}_{L^{6}}/3,
 \\ \label{Eqn:RelNrjK}
 &G(f)  := E(f) - K(f)/6
  =  \|\na f \|^{2}_{L^2}/3.
\end{align}
It follows from \eqref{Eqn:SobIneqMax} and similar arguments to ~\cite{nakschlagkg} that  
\begin{align}
 E(W) & = \inf \{ E(\fy) \mid 0\neq \fy \in \dot{H}^{1}, K(\fy)=0 \}
  \label{Eqn:Charac2GdState} \\
 &= \inf \{ G(\fy) \mid 0 \neq\fy \in \dot{H}^{1}, K(\fy) \leq 0 \}
  \label{Eqn:Charac3GdState} \\
 & = \inf \{ I(\fy) \mid  0\neq\fy \in \dot{H}^{1}, K(\fy) \leq 0 \}
  \label{Eqn:Charac4GdState}
\end{align}

Let $S^{\sigma}_{a}$ be the one-parameter group of dilation operators defined as follows
\EQ{
 S^{\sigma}_{a} f(x) & := e^{(3/2 + a) \sigma} f(e^{\sigma} x),}
and let $S'_{a}:= \partial_{\sigma} S^{\sigma}_{a}|_{\sigma=0}$ be its generator. It is easy to see that the adjoint is given by $(S_{a}^{\sigma})^{*}=S^{-\si}_{-a}$, hence by differentiating in $\si$,
\EQ{ \label{Eqn:AdjointS}
 (S'_{a})^{*} = - S'_{-a}.}

We denote by $S$, $W$, $\bar S$ and $N$ the following mixed $L^p$ spaces on $\R^{1+3}$
\EQ{
 \pt S:= L_{t,x}^{10}(\R^{1+3}),
 \pq W := L_{t}^{10}(\R;L_{x}^{30/13}(\R^3)),
 \pr \bar{S}:= L_{t,x}^{10/3} (\R^{1+3}),
  \pq N := L_{t}^{2}(\R;L_{x}^{6/5}(\R^3)).}
We also use the homogeneous Sobolev spaces defined by completion of the Schwarz space with respect to the norm
\EQ{
 \|v\|_{X^1}:=\|\na v\|_X\pq (X=W,\bar{S},N).}
For any interval $J\subset\R$ and any function space $X$ on $\R^{1+3}$, the restriction of $X$ onto $J$ is denoted by $X(J)$.
The Sobolev embedding implies
\EQ{ \label{Eqn:RelSJWJ}
 \| v \|_{S(J)} & \lesssim \|v \|_{W^1(J)}.}
We recall the $L^p$ decay and the Strichartz estimates (see e.g., \cite{keeltao}). For any $p\in[2,\I]$
\EQ{ \label{Lp decay}
 \|e^{it\De}\fy\|_{L^p_x} \lec |t|^{-3(1/2-1/p)}\|\fy\|_{L^{p'}_x},}
where $p'=p/(p-1)$, and for any interval $J\ni 0$,
\EQ{ \label{Strich}
 \pt\|e^{-it\De}\fy\|_{(L^\I_tL^2_x\cap \bar{S} \cap W)(\R)}
 \lec \|\fy\|_{L^2},
 \pr\left\|\int_0^te^{-i(t-s)\De}F(s)ds\right\|_{(L^\I_tL^2_x \cap \bar{S} \cap W)(J)} \lec \|F\|_{N(J)}.}

In this paper, we constantly use the linearized operator $\mathcal{L}$ defined by
\EQ{
 \mathcal{L} f  = L_{+} f_{1} + i L_{-} f_{2},
  \pq L_{+} := - \De - 5 W^{4},
  \pq L_{-} :=  - \De - W^{4}.}
We recall some spectral properties of $ \mathcal{L}$ (see \cite{duymerle,gril,schlag}):
\begin{itemize}
\item It has two resonance functions $i W$ and $W'$.
\item It has two simple eigenvalues $\pm \mu$ (with $\mu >0$) and
two smooth, exponentially decaying eigenfunctions $g_{\pm} =g_{1} \mp i g_{2}$
that satisfy $i \mathcal{L} g_{\pm}= \pm \mu g_{\pm}$.
In other words, $L_+g_1=-\mu g_2$ and $L_-g_2=\mu g_1$.
\item $L_{-} \geq 0$ and $\Ker(L_{-})=\Span\{W\}$ on $\dot H^1\rad$.
\item $\Ker(L_{+}) = \Span\{W'\}$ on $\dot H^1\rad$.
\end{itemize}
Fix a real-valued radial function $\chi \in \cS\rad(\R^3)\subset L^{6/5}(\R^3)$ such that
\EQ{ \label{Eqn:ChoiceW}
 \pt \LR{W',\chi} \neq 0,\pq  \LR{W,\chi}\neq 0,
 \pq \LR{\chi,g_{1}} = \LR{\chi,g_{2}} = 0.}
To see existence of such $\chi$, suppose for contradiction that $\{g_1,g_2\}^\perp$ is included in $\{W'\}^\perp$ or $\{W\}^\perp$, where $A^\perp:=\{\fy\in\cS(\R^3)\mid A\ni\forall\psi,\ \LR{\fy,\psi}=0\}$.
Since $\cS$ is reflexive, it implies that either $W'$ or $W$ is in $\Span\{g_1,g_2\}$, contradicting the slow decay of $W',W$ by the rapid decay of $g_1,g_2$.
Therefore $\{g_1,g_2\}^\perp$ is not included in $\{W'\}^\perp$ nor $\{W\}^\perp$.
The same conclusion holds under the radial restriction, since orthogonality against radial functions is determined by the spherical average.
Hence there exists $\chi\in\cS\rad(\R^3)$ satisfying \eqref{Eqn:ChoiceW}.

\section{Proof of the main theorem}
\label{Section:ThmMain}
The proof of Theorem \ref{Thm:Main} relies upon some propositions stated below.
The first proposition, proved in Section \ref{Section:ContOrth}, gives a decomposition of a vector $\fy \in \dot{H}^{1}$
close to the ground states $\cW$, taking account of two parameters (the rotation parameter and the scaling parameter) and a constraint (the so-called orthogonality condition).

\begin{prop}[Orthogonal decomposition of $\fy$] \label{Prop:OrthDecompu}
There exist an absolute constant $0<\de_E\ll 1$ and a $C^1$ function $(\ti\te,\ti\si):B_{\de_E}(\cW)\to(\R/2\pi\Z)\times\R$ with the following properties.
For any $\fy \in B_{\de_E}(\cW)$, putting
\EQ{ \label{Eqn:Decomp1}
 \fy = e^{i\ti\te(\fy)} S_{-1}^{\ti\si(\fy)} (W + v),}
we have
\EQ{ \label{Eqn:Decomp2}
 \pt (v | \chi) =  0,
 \pq d_\cW(\fy) \sim \|v\|_{\dot H^1}.}
Moreover $(\ti\te(\fy),\ti\si(\fy))\in(\R/2\pi\Z)\times\R$ is unique for the above property.
Furthermore, if $\|\fy-W_{\te,\si}\|_{\dot H^1}\ll 1$ for some $(\te,\si)\in\R^2$, then
\EQ{ \label{Eqn:Decomp3}
  | (e^{i\ti\te(\fy)} - e^{i\te}, \ti\si(\fy)-\si)| \lec \| \fy - W_{\te,\si}\|_{\dot H^1}.}
\end{prop}
The second proposition, proved in Section \ref{Section:LinearParamW}, describes more precisely the decomposition in Proposition \ref{Prop:OrthDecompu}, taking
into account the spectral properties of $\mathcal{L}$. This decomposition does not use the radial symmetry.

\begin{prop}[Spectral decomposition of $v$] \label{Prop:DecompSpectv}
For any $v\in\dot H^1$, there exists a unique decomposition
\EQ{ \label{Eqn:Decompv}
  \pt v  = \lambda_{+}  g_{+} + \lambda_{-}  g_{-} + \gamma,
 \pq \la_\pm\in\R,\pq \ga\in\dot H^1,}
such that $\omega (g_{\pm}, \gamma) =0$. 
After normalizing $g_\pm$ (or $g_1$ and $g_2$) such that 
\EQ{   \label{Eqn:CondOmega1}
 \pt \omega(g_{+},g_{-}) = 2\LR{g_1,g_2}= 1,
 \pq \pm\om(W,g_\mp)=\LR{W,g_2}>0,}
the above decomposition is given by 
\EQ{ \label{Eqn:DecompLambdaZero}
 \la_\pm := \pm\om(v,g_\mp).}
Putting $\la_1:=(\la_++\la_-)/2$ and $\la_2:=(\la_+-\la_-)/2$, it can also be written as
\EQ{ \label{Eqn:DecompLambdaOne}
 \pt v  = 2\la_1 g_{1} - 2i \lambda_{2}  g_{2} + \gamma, \pq \la_1=  \LR{v_1|g_2}, \pq \la_2=-\LR{v_2|g_1},}
with $\LR{\ga_1|g_2}=\LR{\ga_2|g_1}=0$.

\end{prop}

The third proposition, also proved in Section \ref{Section:LinearParamW}, aims at describing the dynamics of the solution near the ground states, using the decomposition in Proposition \ref{Prop:OrthDecompu}. Again, this does not use the radial symmetry.

\begin{prop}[Linearization and parametrization around $\cW$] \label{Prop:LinearParamW}
Let $u$ be a solution of \eqref{Eqn:SchrodCrit} on an interval $I$ in the form \eqref{Eqn:Decomp1}, i.e., $(\te,\si,v):I\to(\R/2\pi\Z)\times\R\times\dot H^1$ is defined by 
\EQ{ \label{Eqn:Decompu}
 u(t) = e^{i \theta(t)} S_{-1}^{\sigma (t)} (W + v(t)),
 \pq \te(t):=\ti\te(u(t)),\pq \si(t):=\ti\si(u(t)).}
Then, letting $\tau:I\to\R$ such that $\tau'(t):= e^{2 \sigma (t)}$, we have
\EQ{ \label{Eqn:Linearization}
 \partial_\ta v  & = i \mathcal{L} v  - (i \theta_\ta  + \sigma_\ta  S'_{-1}) (W + v ) -i N(v),}
where $\te_\ta=\frac{\p \te}{\p \ta}$ etc., and 
\EQ{
 N(f) \pt:= |W + f|^{4} (W + f) - W^{5} - \partial_{\lambda}|_{\lambda  =0}
\left( |W + \lambda f|^{4} (W + \lambda f) \right).}
Furthermore,  $\p_\ta(\te,\si)=O(\|v\|_{\dot H^1})$ and, decomposing $v$ by Proposition \ref{Prop:DecompSpectv}, 
\EQ{ \label{Eqn:EstDynLambda}
 \partial_\ta \lambda_{\pm}  & = \pm \mu \lambda_{\pm}  + O ( \| v  \|^{2}_{\dot H^1}), }
or equivalently,
\EQ{ \label{Eqn:EstDynLambdaOne}
 \pt\partial_\ta \lambda_{1} = \mu \lambda_{2}  + O ( \|v  \|^{2}_{\dot H^1}),
 \pq\partial_\ta \lambda_{2} = \mu \lambda_{1}  + O ( \|v  \|^{2}_{\dot H^1}).}
\end{prop}

The next proposition, proved in Section \ref{Section:ContOrth}, shows that the orthogonal direction $\gamma$ of $v$ in \eqref{Eqn:Decompv} 
can be controlled by the linearized energy:

\begin{prop}[Control of orthogonal direction] \label{Prop:ControlOrth}
For any function $w\in\dot H^1\rad$ satisfying $\LR{w_{1}, g_{2}}=0$,  we have
\EQ{ \label{Eqn:ControlOrth}
 \| \nabla w \|^{2}_{L^{2}} & \sim |(w | \chi)|^2 + \LR{\cL w,w}. }
\end{prop}

Hence in the subspace $\{v\in\dot H^1\mid (v|\chi)=0\}$, 
we can define an equivalent norm $E$ using the decomposition of Proposition \ref{Prop:DecompSpectv}
\EQ{ \label{Eqn:VeVH1}
 \| v  \|^{2}_{E}  & := \mu \left( \lambda_{1}^{2}  + \lambda_{2}^{2}  \right) + \frac{1}{2}\LR{\mathcal{L} \gamma, \gamma} \sim \la_1^2+\la_2^2+\|\ga\|_{\dot H^1}^2 \sim \|v\|_{\dot H^1}^2. }
In particular, in the decomposition of Proposition \ref{Prop:OrthDecompu}, we have 
\EQ{
 d_\cW(\fy) \sim \|v\|_{\dot H^1} \sim \|v\|_E.}

Henceforth, we assume that whenever a solution $u$ of \eqref{Eqn:SchrodCrit} is in $B_{\de_E}(\cW)$, the coordinates $\si$, $\te$, $v$, $\la_\pm$, $\la_1,\la_2$ and $\ga$ are defined by \eqref{Eqn:Decompu}, \eqref{Eqn:Decompv} and \eqref{Eqn:DecompLambdaOne}, while $\ta(t)$ is a solution of $\dot\ta(t)=e^{2\si}$. In short,
\EQ{ \label{coord around W}
 \pt e^{-i\te}S^{-\si}_{-1}u-W=v=\la_+g_++\la_-g_-+\ga=2\la_1g_1-2i\la_2g_2+\ga,
 \prq 0=(v | \chi)=(\ga | \chi)=\om(g_\pm,\ga)=\LR{g_1,\ga_2}=\LR{g_2,\ga_1},
 \prq \de_E>d_\cW(u) \sim \|v\|_{\dot H^1} \sim \|v\|_E,
 \pq(\te,\si)=(\ti\te(u),\ti\si(u)), \pq \dot\ta=e^{2\si}.}

The next proposition, proved in Section \ref{Section:extend v}, ensures the existence of a solution $u$ of \eqref{Eqn:SchrodCrit} in a neighborhood of $\cW$ as long as the scaling parameter $\si$ is bounded from above. 
\begin{prop}[Uniform local existence in $\ta$] \label{Prop:extend v}
There exists an absolute constant $\de_L\in(0,\de_E/2)$ such that for any solution $u$ of \eqref{Eqn:SchrodCrit} with $d_\cW(u(0))=:\de\in[0,2\de_L]$, we have
$T_\pm(u)>3e^{-2\si(0)}=:T_0$, $\pm(\ta(\pm T_0)-\ta(0))>2$, and for $|t|\le T_0$,
\EQ{
 \de_E>d_\cW(u(t))\sim \de, \pq \si(t)=\si(0)+O(\de).}
\end{prop}

Now we are ready to define the nonlinear distance $\ti d_\cW$. Let $\fy \in B_{\delta_{E}} (\cW)$. Consider the decomposition
\eqref{Eqn:Decomp1} of $\fy$. Then we define a local distance $d_0:B_{\de_E}(\cW)\to[0,\I)$ by
\EQ{ \label{def d0}
 d_{0}(\fy)^2 := E(\fy) - E(W) + 2\mu \la_1^2 .}

As observed in \cite{nakaschlagschrod}, this is close to be convex in $\ta$ when the solution is ejected out of a small neighborhood of $\cW$, but it may have small oscillation around minima in $\ta$.
This is a difference for the Schr\"odinger equation from the Klein-Gordon equation, for which $d_0^2$ is strictly convex (see \cite{nakschlagkg}).
We could treat the possible oscillation as in \cite{nakaschlagschrod} by waiting for a short time before the exponential instability dominates, which would however bring a certain amount of complication to the statements as well as the proof.

Here instead, we introduce a dynamical mollification of $d_0^2$, which yields a strictly convex function in $\ta$.
The same argument works in the subcritical setting as in \cite{nakaschlagschrod}.
Let $u$ be the solution of \eqref{Eqn:SchrodCrit} with  initial data $ u(0) := \fy \in B_{2\de_L}(\cW)$.
Then Proposition \ref{Prop:extend v} ensures that $u$ exists at least for $|\ta-\ta(0)|\le 2$ in $B_{\de_E}(\cW)$.
Using the decomposition \eqref{coord around W} with $\tau(0):=0$, let
\EQ{ \label{def d1}
 d_1(\fy)^2 := \int_\R \phi(\ta)d_0(u)^2 d\ta,}
where $\phi$ is the cut-off function in \eqref{def phi}. 
This defines the function $d_1:B_{2\de_L}(\cW)\to[0,\I)$.
Then, we define the nonlinear distance function $\ti d_\cW:\dot H^1\rad\to[0,\I)$ by
\EQ{ \label{Eqn:NonlinearDistDef}
  \tilde{d}_{\cW}(\fy) := \phi_{\de_L}(d_{\cW}(\fy))  d_1(\fy) + \phi^C_{\de_L}(d_{\cW} (\fy)) d_{\cW}(\fy).}
The following proposition, proved in Section \ref{Section:NonlinearDist}, gives the main static properties of the distance function.

\begin{prop}[Nonlinear distance function] \label{Prop:DistFuncEigen}
The functional $\ti d_\cW$ on $\dot H^1\rad$ is invariant for the rotation and scaling, and equivalent to $d_\cW$.
Precisely, there exists an absolute constant $C\in(1,\I)$ such that for all $\fy\in\dot H^1\rad$ and $(\al,\be)\in\R^2$,
\EQ{
 d_\cW(\fy)/C \le \ti d_\cW(\fy)=\ti d_\cW(e^{i\al}S_{-1}^\be\fy) \le Cd_\cW(\fy).}
Moreover, there exists an absolute constant $c_D\in(0,1)$ such that putting
\EQ{ \label{def ticH}
 \ck\cH := \{ \fy \in \dot H^1\rad \mid E(\fy) < E(W) + (c_D\tilde{d}_{\cW}(\fy))^2\}}
we have
\EQ{ \label{Eqn:DomEigenMode}
 \fy \in  B_{\de_L}(\cW) \cap \ck\cH \implies \tilde{d}_{\cW}(\fy)  \sim  |\lambda_1|.}
\end{prop}

Hence we can use $\ti d_\cW(\fy)$ to measure the distance to $\cW$, instead of the standard $d_\cW(\fy)$. The $\de$ neighborhood with respect to this distance function is denoted by
\EQ{
 \ti B_\de(\cW):=\{\fy \in \dot H^1\rad \mid \ti d_\cW(\fy)<\de\}.}

The next proposition, proved in Section \ref{Section:Ejection}, describes the dynamics close to the ground states in the ejection mode:

\begin{prop}[Dynamics in the ejection mode] \label{Prop:DynEjection}
There is an absolute constant $\de_X\in(0,1)$ such that $\ti B_{\de_X} (\cW) \subset B_{\de_L} (\cW)$, and that for any solution $u$ of \eqref{Eqn:SchrodCrit} with
\EQ{ \label{Eqn:EjecScplus}
 u(t_0) \in \ti B_{\de_X} (\cW) \cap\ck\cH \pq \text{and} \pq \p_t\ti d_\cW(u(t_0)) \ge 0,}
at some $t_0\in I(u)$, we have the following. $\ti d_\cW(u(t))$ is increasing until it reaches $\de_X$ at some $t_X\in(t_0,T_+(u))$. For all $t\in[t_0,t_X]$, we have
\begin{align}
 \pt \tilde{d}_{\cW} (u(t)) \sim |\lambda_{1}(t)| \sim e^{\mu(\ta(t)-\ta(t_0))}\ti d_\cW(u(t_0)),
   \label{Eqn:Dyndq}
 \pr \| \gamma (t) \|_{\dot{H}^{1}} \lec \ti d_\cW(u(t_0)) + \ti d_\cW(u(t))^2,
   \label{Eqn:Dynvperp}
 \pr |(e^{i\te(t)}-e^{i\te(t_0)}, \sigma(t) - \sigma(t_0)) |\lec \ti d_\cW(u(t)),
   \label{Eqn:Dyntesi}
\end{align}
$\sign(\lambda_{1}(t))$ is constant, and there exists an absolute constant $C_K>0$ such that
\EQ{ \label{Eqn:DynK}
 -\sign(\lambda_{1}(t)) K(u(t)) & \gtrsim (e^{\mu(\ta(t)-\ta(t_0))} - C_K )\ti d_\cW(u(t_0)).}
\end{prop}
\begin{rem}
By time-reversal symmetry\footnote{If $u(t,x)$ is a solution of \eqref{Eqn:SchrodCrit}, then $\overline{u(-t,x)}$ is also a solution of \eqref{Eqn:SchrodCrit}.}, a similar result holds in the negative time direction, where the last condition of \eqref{Eqn:EjecScplus} is replaced with $\p_t\ti d_\cW(u(t))\le 0$.
\end{rem}

The next proposition gives a variational estimate away from the ground states, and it is a consequence of \eqref{Eqn:Charac2GdState}, cf.~\cite{nakschlagwave,nakschlagbook}.
\begin{prop}[Variational estimates] \label{Prop:VarEst}
There exist two increasing functions $\ep_V$ and $\ka$ from $(0,\I)$ to $(0,1)$,  and an absolute constant $c_V>0$, such that for any
$ \fy \in \dot H^1\rad$ satisfying $E(\fy)< E(W)+\ep_V(\ti d_\cW(\fy))^2$, we have
\EQ{
 K(\fy) \ge \min(\ka(\ti d_\cW(\fy)), c_V \|\nabla \fy \|^{2}_{L^{2}})
 \text{ or }
 K(u) & \leq -\ka(\ti d_\cW(\fy)).}
\end{prop}

The next proposition, proved in Section \ref{Section:Sign}, defines a functional $\Te$ that decides the fate of the solution around $t=T_\pm(u)$, as well as at the exit time $t=t_X$ in the above proposition.

\begin{prop}[Sign functional] \label{Prop:SignProp}
There exist an absolute constant $\ep_S\in(0,1)$ and a continuous function $\Te:\cH^{\ep_S}\cap\ck\cH\to\{\pm 1\}$, such that for some $0<\de_1<\de_2<\de_X$ and for any $\fy \in \cH^{\ep_S}\cap\ck\cH$, with the convention $\sign 0=+1$,
\EQ{ \label{def Te}
  \CAS{\ti d_\cW(\fy)\ge\de_1 \implies \Te(\fy)=\sign K(\fy),\\
   \ti d_\cW(\fy)\le\de_2 \implies \Te(\fy)=-\sign \la_1,}}
and $\Te(e^{i\al}S_{-1}^\be \fy)=\Te(\fy)$ for all $(\al,\be)\in\R^2$. Moreover, if $E(\fy)<E(W)$ then $\Te(\fy)=\sign K(\fy)$.
\end{prop}
Note that the region $\{\fy\in\cH^{\ep_S}\cap\ck\cH\mid \Te(\fy)=+1\}$ is bounded in $\dot H^1$, because
\EQ{ \label{unif bd Te+}
 \CAS{K(\fy)\ge 0 \implies \|\na \fy\|_2^2\le 3E(\fy)\le 3(E(W)+\ep_S^2),
 \\ \ti d_\cW(\fy)\le \de_X \implies \|\fy\|_{\dot H^1}\le \|W\|_{\dot H^1}+C\de_X,}}
but it does not imply global existence for solutions staying in this region, because of the critical nature of \eqref{Eqn:SchrodCrit}.

The continuity of $\Te$ implies that for any solution $u$ in $\cH^\ep\subset\cH^{\ep_S}$, $\Te(u)\in\{\pm 1\}$ can change along $t\in I(u)$ only if $u$ goes through the small neighborhood $\cH^\ep\setminus\ck\cH\subset \ti B_{\ep/c_D}(\cW)$.
The next proposition, proved in Section \ref{Section:OnePassLemma}, implies that such a transition can happen at most once for each solution.
\begin{prop}[One-pass] \label{Prop:OnePassLemma}
There exist an absolute constant $\de_B\in(0,\de_X)$ and an increasing function $\ep_B:(0,\de_B]\to(0,\ep_S]$ satisfying $\ep_B(\de)<c_D\de$ for $\de\in(0,\de_B]$, and for any solution $u$ of \eqref{Eqn:SchrodCrit} with $u(t_0)\in\cH^{\ep_B(\de)}\cap\ti B_{\de}(\cW)$ at some $t_0\in I(u)$,
\EQ{ \label{no return}
 \exists t_+(\de)\in(t_0,T_+(u)],\text{s.t. }\CAS{t_0\le t<t_+(\de) \implies \ti d_\cW(u(t))<\de, \\
 t_+(\de)<t<T_+(u) \implies \ti d_\cW(u(t))>\de.}}
\end{prop}
\begin{rem} \label{Rem:OnePassLemma}
By time-reversal symmetry, there also exists $t_-\in[-T_-(u),t_0)$ such that $\ti d_\cW(u(t))<\de$ for $t_-<t<t_0$ and $\ti d_\cW(u(t))>\de$ for $-T_-(u)<t<t_-$.
\end{rem}
\begin{rem} \label{in ckcH}
$\ep_B(\de)\le \ep_S$ and $\ep_B(\de)<c_D\de$ imply that $\cH^{\ep_B(\de)}\setminus\ti B_\de(\cW)\subset\cH^{\ep_S}\cap \ck\cH$, where $\Te$ is defined by Proposition \ref{Prop:SignProp}.
\end{rem}
The above proposition tells that if a solution gets out of $\ti B_\de(\cW)$, then it can never return there.
Moreover, it applies to all $\de\in(0,\de_B]$ satisfying $E(u)<E(W)+\ep_B(\de)^2$.
The solution $u$ stays around $\cW$ iff $t_+(\de)=T_+(u)$. The following proposition, proved in Section \ref{Sect:staying}, gives more precise description of such solutions.
\begin{prop} \label{Prop:staying}
Under the assumption of Proposition \ref{Prop:OnePassLemma}, suppose that $t_+(\de)=T_+(u)$.
Then there exists $t_1\in[t_0,T_+(u)]$ such that $\ti d_\cW(u(t))$ is decreasing on $[t_0,t_1)$, and $u(t)\not\in\ck\cH$ for all $t\in[t_1,T_+(u))$.
If $t_1=T_+(u)$, then $\ti d_\cW(u(t))\searrow 0$ as $t\nearrow T_+(u)$, which implies $E(u)=E(W)$. 
We have similar statements in the case $t_-(\de)=T_-(u)$ by the time-reversal symmetry. 
\end{prop}

The next proposition, proved in Section \ref{Section:LongTimeFarGd}, describes the asymptotic behavior of solutions which are away from the ground states.
\begin{prop}[Asymptotic behavior] \label{Prop:FarFromGd}
There exists an increasing function $\ep_*:(0,\de_B]\to(0,\ep_S]$ with $\ep_* (\delta) \le \ep_B (\delta)$ for $\delta \in (0,\delta_{B})$ and the
following properties. Suppose that $u$ is a solution of \eqref{Eqn:SchrodCrit} satisfying $u([t_0,T_+(u)))\subset \cH^{\ep_*(\de)}\setminus \ti B_\de(\cW)$ for some $t_0\in I(u)$.
If $\Te(u(t))=+1$ at some $t\in[t_0,T_+(u))$, then $T_+(u)=\I$, and $u$ scatters as $t \to \infty $.
If $\Te(u(t))=-1$ and $u(t)\in L^2_x$ at some $t\in[t_0,T_+(u))$, then $T_{+}(u) < \infty$.
By time-reversal symmetry, the same statements hold for the negative time direction $(-T_-(u),t_0]$.
\end{prop}
Note that by Remark \ref{in ckcH} and $\ep_*\le\ep_B$, $\Te(u(t))$ is well defined for all $t\in[t_0,T_+(u))$ in the above statement.

Armed with the above propositions, we are now ready to prove Theorem \ref{Thm:Main}.
\begin{proof}[Proof of Theorem \ref{Thm:Main}]
Using $c_D,\de_B,\ep_*$ in Propositions \ref{Prop:DistFuncEigen}, \ref{Prop:OnePassLemma} and  \ref{Prop:FarFromGd}, define
\EQ{
 \ep_\star:=\ep_*(\de_B)>0, \pq \cX_\ep:=\{\fy \in \cH^\ep \mid c_D\ti d_\cW(\fy)\le\ep\}.}
Then $\cX_\ep$ is relatively closed in $\cH^\ep$, and $\cW\subset\cX_\ep\subset \ti B_{2\ep/c_D}(\cW)$.
Since $\ep_\star\le\ep_S$ and
\EQ{
 \fy \in \cH^\ep\setminus\cX_\ep
 \implies E(\fy)-E(W)<\ep^2<(c_D\ti d_W(\fy))^2 \implies \fy\in\ck\cH,}
the functional $\Te$ is defined by Proposition \ref{Prop:SignProp} on $\cH^\ep\setminus\cX_\ep$ for all $\ep\in(0,\ep_\star]$.

Next we consider the dynamics.
Let $\ep\in(0,\ep_\star]$ and let $u\in\cH^\ep$ be a solution of \eqref{Eqn:SchrodCrit}. Let $I_0(u)$ be as in the theorem and let $I_C(u):=I(u)\setminus I_0(u)$.

Take any $\de\in(0,\de_B]$ satisfying $\ep\le \ep_*(\de)$.
First suppose that there exists $t_0\in I_C(u)$ such that
\EQ{ \label{ejection}
 \ti d_\cW(u(t_0))<\de, \pq \p_t\ti d_\cW(u(t_0))\ge 0.}
Then by Proposition \ref{Prop:DynEjection}, $\ti d_\cW(u(t))$ is increasing until it reaches $\de_X$ at some $t_X\in(t_0,T_+(u))$.
Since $\ti d_\cW(u(t_0))<\de<\de_X$, there exists $t'\in(t_0,t_X)$ such that $\ti d_\cW(u(t'))=\de$.
Then Proposition \ref{Prop:OnePassLemma} implies that $\ti d_\cW(u(t))>\de$ for all $t\in(t',T_+(u))$.
Hence $\ti d_\cW(u(t))\ge \ti d_\cW(u(t_0))$ for all $t\in[t_0,T_+(u))$, which implies $[t_0,T_+(u))\subset I_C(u)$.
Then by Proposition \ref{Prop:FarFromGd}, $\Te(u(t))$ on $[t_0,T_+(u))$ decides the behavior of $u$ towards $T_+(u)$.

If the last condition of \eqref{ejection} is replaced with $\p_t\ti d_\cW(u(t_0))\le 0$, then the time reversed version of the above argument implies that $(-T_-(u),t_0)\subset I_C(u)$ and the behavior of $u$ towards $-T_-(u)$ is determined by $\Te(u(t))$ there.

Next consider the case where there exist $t_1\in I_0(u)$ and $t_2\in I_C(u)$.
Suppose that $t_1<t_2$.
Since $\ti d_\cW(u)\le \ep/c_D< \de$ on $I_0(u)$, we may assume $\ti d_\cW(u(t_2))<\de$ by decreasing $t_2$ if necessary.
Then the above argument works with $t_0:=t_2$, either forward or backward in time, but the latter case leads to a contradiction with the existence of $t_1\in I_0(u)$ smaller than $t_2$. Hence we have $\p_t\ti d_\cW(u(t_2))\ge 0$ and $[t_2,T_+(u))\subset I_C(u)$. If $t_1>t_2$, then in the same way, we deduce that $(-T_-(u),t_2]\subset I_C(u)$. Therefore, $I_0(u)$ is either empty or an interval.

Concerning the behavior of $u$ towards $T_+(u)$, it only remains to consider the following case: $I(u)=I_C(u)$ but \eqref{ejection} is never satisfied by any $t_0\in I(u)$.
In this case, there are only two possibilities: either $\ti d_\cW(u(t))\ge\de$ all over $I(u)$, or $\ti d_\cW(u(t))$ goes below $\de$ and then stays there.
In the former case, we can apply Proposition \ref{Prop:FarFromGd} to decide the behavior around $T_\pm(u)$.
In the latter case, we can apply Proposition \ref{Prop:staying} on some interval $[t_0,T_+(u))$ where $\ti d_\cW(u)<\de$. Then
\EQ{
 \limsup_{t\nearrow T_+(u)}c_D\ti d_\cW(u(t))\le \sqrt{E(u)-E(W)}<\ep,}
contradicting $I(u)=I_C(u)$.
This completes the investigation around $T_+(u)$, and the behavior towards $-T_-(u)$ is treated in the same way.
Theorem \ref{Thm:Main} is proved.
\end{proof}

\begin{rem}
The same argument as above works if we replace $X_\ep$ with
\EQ{
 \ti X_\ep := \cH^\ep \setminus \ck\cH = \{\fy\in \cH^\ep \mid E(u)\ge E(W)+(c_D\ti d_\cW(u))^2\},}
which is smaller and essentially independent of $\ep$.
In that case, however, we need to modify our conclusion for the special solutions $W^\pm$ constructed by Duyckaerts and Merle \cite{duymerle} on the threshold $E(u)=E(W)$, namely those two solutions (unique modulo the invariance) which are exponentially convergent to $W$ as $t\to\I$, and scattering or blowing up in $t<0$.
These solutions are in $\cH^\ep\cap\ck\cH$ for all $t\in I(u)$ and $\ep>0$, where $\Te=\pm 1$ according to its behavior in $t<0$.
Thus $\Te$ fails to give the correct prediction for $t>0$ in this case.
This is exactly the case $t_0=t_1=T_+(u)$ in Proposition \ref{Prop:staying}, namely $\ti d_\cW(u)\searrow 0$ as $t\nearrow T_+(u)$.
The classification in \cite{duymerle} also implies that it happens only for those special solutions.
In other words, $X_\ep$ has been enlarged from $\ti X_\ep$ in order to eliminate those solutions.
\end{rem}

\section{Orthogonal decomposition} \label{Section:ContOrth}
In this section, we prove Proposition \ref{Prop:OrthDecompu}.
Define a $C^1$ function $F:\R\times\R\times\dot H^1\rad\to\C$ by
\EQ{
 F(\theta,\sigma,\psi) & := (e^{-i \theta} S_{-1}^{-\sigma}(W + \psi) - W | \chi).}
Since $F(0,0,0)=0$ and (writing $F=F_1+iF_2=(F_1,F_2)\in\R^2$)
\EQ{
 \partial_{\theta, \sigma} F(0,0,0) &= \MAT{
  \partial_{\theta} F_{1}(0,0,0) & \partial_{\sigma} F_{1}(0,0,0) \\
  \partial_{\theta} F_{2}(0,0,0) & \partial_{\sigma} F_{2}(0,0,0)}
 \pn = \MAT{0 &  -\LR{W',\chi} \\ - \LR{W,\chi} & 0},}
the implicit function theorem yields $\de>0$ and a $C^{1}$ function $(\theta,\sigma):B_\de(0)\to\R^2$, where $B_\de(0)$ denotes the $\de$ neighborhood of $0$ in $\dot H^1\rad$, such that $F(\theta(\psi),\sigma(\psi),\psi)=0$ and $\te(0)=\si(0)=0$, which is unique in $B_\de(0)$ and a neighborhood of $0\in\R^2$.
For any $\fy \in B_\de(\cW)$, there exists $(\al,\be,\psi)\in(\R/2\pi\Z)\times\R\times B_\de(0)$ such that $\fy=e^{i\al} S_{-1}^\be(W+\psi)$. Then we put
\EQ{
 \ti\te(\fy):=\te(\psi)+\al\in \R/2\pi\Z, \pq \ti\si(\fy):=\si(\psi)+\be\in\R.}
Then defining $v$ by \eqref{Eqn:Decomp1}, we have $v=e^{-i\te(\psi)}S_{-1}^{-\si(\psi)}(W+\psi)-W$ and so
\EQ{
 \pt (v|\chi)=F(\te(\psi),\si(\psi),\psi)=0,
 \pr \|v\|_{\dot H^1}\lec|\te(\psi)|+|\si(\psi)|+\|\psi \|_{\dot H^1} \lec \|\psi \|_{\dot H^1}.}
This implies \eqref{Eqn:Decomp3}, as well as $d_\cW(\fy)\sim\|v\|_{\dot H^1}$, choosing $(\al,\be)$ such that
$d_\cW(\fy) \sim \| \fy  - e^{i \alpha} S_{-1}^{\beta} W \|_{\dot H^1}$.

To see the uniqueness of $(\ti\te,\ti\si)$ for each $ \fy \in B_\de(\cW)$, suppose that we have two ways of decomposition
\EQ{
 \fy=e^{i\al_1}S_{-1}^{\be_1}(W+v)=e^{i\al_2}S_{-1}^{\be_2}(W+v')}
with \eqref{Eqn:Decomp2} for both $v$ and $v'$, then putting $(\al,\be):=(\al_2,\be_2)-(\al_1,\be_1)$,
\EQ{ \label{orth diff}
 0=(v-v'|\chi)=((e^{i\al}S_{-1}^{\be}-1)(W+v')|\chi).}
Since $W\in\dot H^1$ and $\chi\in\dot H^{-1}$, we have
\EQ{
 \|(e^{i\al}S_{-1}^\be-1)W\|_{\dot H^1}\sim \min(|e^{i\al}-1|+|\be|,1)
 \gec \|(e^{i\al}S_1^\be-1)\chi\|_{\dot H^{-1}}.}
Then by \eqref{orth diff}, we obtain 
\EQ{
 0=(1+O(\|v'\|_{\dot H^1}))(|e^{i\al}-1|+|\be|),}
hence the uniqueness of $(\ti\te,\ti\si)$. 
\qedsymbol

\section{Evolution around the ground states} \label{Section:LinearParamW}
In this section, we prove Proposition \ref{Prop:DecompSpectv} and Proposition \ref{Prop:LinearParamW}.

First we show that we can normalize  $g_{+}$ and $g_{-}$ such that \eqref{Eqn:CondOmega1} holds.
Since $L_{-}\ge 0$ with $\Ker(L_{-})=\Span\{W\}$, we have
\EQ{
 0<c:=\LR{L_{-} g_{2},g_{2}}=-\LR{L_+g_1,g_1}=\mu\omega (g_{+},g_{-})/2.}
Hence it is enough to show  $\LR{W , g_{2}}\neq 0$ (see \cite{nakaschlagschrod} for a similar argument).
Suppose for contradiction that $\LR{W,g_2}=0$.
Then $4\LR{W^{5},g_{1}}=-\LR{L_+W,g_1}=\LR{W,\mu g_2}=0$.
Let $0<\de\ll 1$ and $v = \al W + \delta g_{1}$ with $\al =O(\delta^{2})$ to be chosen shortly.
By expansion of the energy
\EQ{
 E(W + v) = E(W) + \frac{1}{2} \LR{L_{+} v,v} + O ( \| v \|^{3}_{\dot{H}^{1}})
 < E(W) - \frac c2\delta^{2} + O (\delta^{3})}
and by expansion of $K$
\EQ{
 K(W + v) & = - 4\LR{W^5, v} + O (\| v \|^{2}_{\dot{H}^{1}})
 = - 4\al \| W \|^{6}_{L^{6}}  + O(\de^2),}
so that one can find $\al=O(\delta^{2})$ such that $K(W+v)=0$, which contradicts \eqref{Eqn:Charac2GdState}.
Hence $\LR{W,g_2}\not=0$ and we can normalize $g_\pm$ such that \eqref{Eqn:CondOmega1} holds.
Then \eqref{Eqn:DecompLambdaZero} and \eqref{Eqn:DecompLambdaOne} are immediate consequences.
Thus we obtain Proposition \ref{Prop:DecompSpectv}.

Next, injecting the decomposition \eqref{Eqn:Decompu} into the equation \eqref{Eqn:SchrodCrit}, we obtain, after straightforward computations using that $W_{\te,\si}$ is a real-valued stationary solution,
\EQ{ \label{Eqn:LinearFirst}
 v_t & = i e^{2 \sigma} \mathcal{L} v - (i \te_t + \si_t S'_{-1}) ( W + v) - i e^{2 \sigma} N(v).}
Applying the change of variable $t\mapsto\ta$ with $\dot\ta = e^{2 \sigma}$ to the above yields \eqref{Eqn:Linearization}. 

Next we consider the equations for the parameters.
Using \eqref{Eqn:Linearization}, we have
\EQ{
 \partial_\ta \lambda_{+}  = \omega ( \partial_\ta v,g_{-} )
 \pt= \omega ( i \mathcal{L} v, g_{-} ) - \omega \left( (i \te_\ta +  \si_\ta  S_{-1}') W, g_{-} \right)
 \prq+ \LR{ N(v), g_{-}} - \omega \left( (i \te_\ta + \si_\ta  S_{-1}') v, g_{-} \right).}
Using \eqref{Eqn:AdjointS} and $\cL^*=\cL$, we see that
\EQ{
 \pt \omega(i \mathcal{L} v, g_{-}) = \omega (i \mathcal{L} g_{-},v) = \mu \lambda_{+},
 \pq \omega \left( (i \te_\ta + \si_\ta  S_{-1}') W, g_{-} \right)  = 0,
 \pr \omega \left( (i \te_\ta + \si_\ta  S_{-1}')v, g_{-} \right) = - \te_\ta \LR{ v, g_{-}} -\si_\ta \LR{i v, S'_{1} g_{-}}. }
Hence
\EQ{ \label{Eqn:DynLambdaplus}
 \p_\ta \lambda_{+} & = \mu \lambda_{+} + \te_\ta \LR{v, g_{-}} + \si_\ta \LR{i v, S'_{1} g_{-}} + \LR{ N(v), g_{-}}.}
Similarly one finds that
\EQ{ \label{Eqn:DynLambdamin}
 \partial_\ta \lambda_{-} & = - \mu \lambda_{-} - \te_\ta \langle v, g_{+} \rangle  - \si_\ta \LR{ i v, S'_{1} g_{+}} - \LR{ N(v),g_{+}}.}
Next we differentiate with respect to $\tau$ the orthogonality condition $(v|\chi)=0$.
Then plugging \eqref{Eqn:Linearization} into $v_\ta$, and using $(g_\pm,\chi)=0$, we get
\EQ{
 0=(v_\ta,\chi)=(i\cL\ga,\chi)-i \te_\ta \LR{W,\chi} - \si_\ta[\LR{W',\chi}-(v,S_1'\chi)]-(iN(v),\chi).}
Hence, using that $\LR{W,\chi}\LR{W',\chi}\not=0$, we obtain
\EQ{ \label{Eqn:DynTheta}
 \partial_\ta ( \theta, \sigma ) & = O \left( \| \gamma \|_{\dot{H}^{1}} + \| v \|^{2}_{\dot{H}^{1}}   \right).}
Now plugging \eqref{Eqn:DynTheta} into \eqref{Eqn:DynLambdaplus} and \eqref{Eqn:DynLambdamin}, we see that \eqref{Eqn:EstDynLambda} holds.
Thus we obtain Proposition \ref{Prop:LinearParamW}.
\qedsymbol

\section{Control by the linearized energy}
In this section, we prove Proposition \ref{Prop:ControlOrth}.
First we prove that, for any $f\in\dot H^1\rad$,
\EQ{ \label{Eqn:PartialNegLplus}
 \langle f, g_{2} \rangle & = 0 \implies \langle L_{+} f, f \rangle \geq 0.}
If the above fails, then for all $(a,b)\in\R^2$,
\EQ{
 \LR{L_{+}(a f + b g_{1}),af + bg_{1}} \pt= a^{2} \LR{L_{+}f,f} - b^2 \LR{L_- g_2, g_2}
 \pn < 0.}
So $L_{+}$ is negative on a two dimensional subspace, which contradicts the fact that $L_{+}$ has only one negative eigenvalue.
Hence \eqref{Eqn:PartialNegLplus} holds.

We are now in position to prove \eqref{Eqn:ControlOrth}.
\eqref{Eqn:ControlOrth} with the $\gtrsim$ sign is obvious, so we only prove \eqref{Eqn:ControlOrth} with the $\lesssim$ sign.
Assume for contradiction that this is false.
Then there exists a sequence $\{w_n\}_{n \ge 1}\subset \dot{H}^{1}\rad$ such that
$\| w_{n} \|_{\dot{H}^{1}} =1$, $(w_n|\chi) \to 0$, $\LR{w_{n,1},g_{2}} =0$ and
\EQ{ \label{Eqn:CondSeq1}
 \LR{L_{+} w_{n,1}, w_{n,1}} \to 0,
 \pq \LR{L_{-} w_{n,2}, w_{n,2}} \to 0,}
as $n \to \infty$, where $w_n=w_{n,1}+iw_{n,2}$.
Since $w_n$ is bounded in $\dot H^1$, passing to a subsequence if necessary, we may assume that it is weakly converging to some $w\in\dot H^1\rad$. Then
\EQ{ \label{Eqn:Conseq1}
 \pt (w|\chi)= 0, \pq \LR{w_1, g_{2}}=0,
 \pr \LR{L_{+} w_1, w_1} \leq 0,
  \pq \LR{L_{-} w_2, w_2} \leq 0.}
Now \eqref{Eqn:PartialNegLplus} and $L_{-}\ge 0$ imply that $\LR{L_+w_1,w_1}=\LR{L_-w_2,w_2}=0$, and so
\EQ{ \label{Eqn:GammaiBound}
 \|\na w_n\|_{L^2}^2 \to \|\na w\|_{L^2}^2.}
Moreover, $w_1$ and $w_2$ are minimizers for the quadratic forms under the constraint $\LR{w_1,g_2}=0$. Hence $L_+ w_1 = L_- w_2=0$, which implies
\EQ{
 w_1 \in \Span\{W'\} & ,\pq  w_2 \in \Span\{W\}.}
Then by \eqref{Eqn:ChoiceW}, we see that $w=0$, which contradicts \eqref{Eqn:GammaiBound} and $\| w_{n} \|_{\dot{H}^{1}}=1$.
\qedsymbol

\section{Uniform local existence} \label{Section:extend v}
In this section, we prove Proposition \ref{Prop:extend v}.
Let $u(0)\in B_\de(\cW)$ for some constant $\de\in(0,\de_E)$, whose smallness will be required in the following.
By rotation and scaling, we may reduce to the case $\te(0)=\si(0)=0$.
Since $W$ is a global solution, it is a consequence of the local wellposedness in $\dot H^1$ that for $\|u(0)-W\|_{\dot H^1}=d_\cW(u(0))=:\de$ small enough, the solution $u$ exists and remains in $O(\de)$ neighborhood of $W$ for $|t|\le 3$.
Since we can solve the equation back to $t=0$ as well, it implies that $d_\cW(u(t))\sim d_\cW(u(0))=\de$ for $|t|\le 3$.
Then by Proposition \ref{Prop:LinearParamW}, we have
\EQ{
 \dot\ta=e^{2\si}, \pq \dot\si=\dot\ta\si_\ta=e^{2\si}O(\de),}
and $\si(0)=0$. Hence $e^{-2\si}=1+O(\de)$ and $\ta=\ta(0)+(1+O(\de))t$ for $|t|\le 3$.
In particular, $\ta(t)$ reaches $\ta(0)\pm 2$ within $I(u)$, if $\de>0$ is small enough. \qedsymbol

\section{Nonlinear distance function} \label{Section:NonlinearDist}
In this section, we prove Proposition \ref{Prop:DistFuncEigen}.

First we prove $\ti d_\cW\sim d_\cW$.
Since $\ti d_\cW=d_\cW$ for $d_{\cW}(\fy)\ge 2\de_L$, it suffices to consider the case $d_{\cW}(\fy) \leq 2 \delta_{L}$.
Decompose $\fy$ by Propositions \ref{Prop:OrthDecompu} and \ref{Prop:DecompSpectv}. Then we have 
\EQ{ \label{Eqn:ExpEnergy}
 E(\fy)- E(W) & = \frac{1}{2} \langle \mathcal{L} v, v \rangle - C(v)
  = -\mu \lambda_{+} \lambda_{-} + \frac{1}{2} \langle \mathcal{L} \gamma, \gamma \rangle - C(v),}
where we used \eqref{Eqn:CondOmega1}.
Here $C(\cdot)$ denotes the following functional on $\dot H^1$:
\EQ{
 C(v) & : = \frac 16\|W+v\|_{L^6}^6 - \sum_{k=0}^2 \left.\frac{1}{k!}\p_\la^k\frac 16\|W+\la v\|_{L^6}^6\right|_{\la=0}=O(\|v\|_{\dot H^1}^3),}
whose Fr\'echet derivative is $N(v)$.
Hence, using Proposition \ref{Prop:ControlOrth},
\EQ{
 d_0(\fy)^2= \| v \|_E^2 - C(v) = \|v\|_E^2 + O(\|v\|_{\dot H^1}^3) \sim d_\cW(\fy)^2 \cdot}
Then by Proposition \ref{Prop:extend v}, we have $d_1(\fy)\sim d_\cW(\fy)$, and so $\ti d_\cW(\fy)\sim d_\cW(\fy)$.

Next we prove \eqref{Eqn:DomEigenMode}.
If $d_\cW(\fy)<\de_L$ and $E(\fy)-E(W)< (c_{D} d_\cW(\fy))^{2}$, then
\EQ{
 \ti d_\cW(\fy)^{2} \sim d_0(\fy)^2=E(\fy)-E(W)+2\mu\la_1^2,}
and so $\ti d_\cW(\fy)^{2} \lec\la_1^2$.
From \eqref{Eqn:VeVH1} we see that $\ti d_\cW(\fy)^{2} \sim \| v \|_{E}^{2} \gtrsim \lambda_{1}^{2}$.

Finally, we check the invariance for the rotation and scaling.
Let $(\al,\be)\in\R^2$, $\fy \in B_{\de_E}(\cW)$ and let $u$ and $u'$ be the solutions of \eqref{Eqn:SchrodCrit} with the initial data
\EQ{
 u(0)=\fy, \pq u'(0)=e^{i\al}S_{-1}^\be \fy,}
with the decompositions by Proposition \ref{Prop:OrthDecompu} and the rescaled time functions
\EQ{
 \pt u=e^{i\te}S_{-1}^{\si}(W+v), \pq u'=e^{i\te'}S_{-1}^{\si'}(W+v'),
 \pr \dot \ta = e^{2\si}, \pq \dot\ta'=e^{2\si'}, \pq \ta(0)=0=\ta'(0).}
Then the uniqueness of $(\ti\te,\ti\si)$ in Proposition \ref{Prop:OrthDecompu} implies
\EQ{
 (\te',\si')=(\te,\si)+(\al,\be), \pq \ta'=\ta e^{2\be},}
while the invariance of the equation \eqref{Eqn:SchrodCrit} implies
$u'(t) = e^{i\al}S_{-1}^\be u(e^{2\be}t)$.
Hence $v$ is invariant in the rescaled time, namely
\EQ{
 \ta(t)=\ta'(t') \implies v(t)=v'(t'),}
which is inherited by $\la_*$ and $\ga$.
Therefore $d_0$ and $d_1$ are invariant, so is $\ti d_\cW$.

\section{Dynamics in the ejection mode} \label{Section:Ejection}
In this section, we prove Proposition \ref{Prop:DynEjection}.
Let $u$ be a solution in the ejection mode \eqref{Eqn:EjecScplus} at $t=t_0\in I(u)$.
Since $\de_X<\de_L$, Proposition \ref{Prop:extend v} implies that either there is a minimal $\ta_X>\ta_0$ such that $\ti d_\cW(u)=\de_X$ at $\ta=\ta_X$, or $\ti d_\cW(u)<\de_X$ for all $\ta\in(\ta_0,\I)$.
Let $\ta_X:=\I$ in the latter case.
Then in both the cases, we have $\ti d_\cW(u)<\de_X$ for $\ta\in(\ta_0,\ta_X)$.
Choosing $\de_X$ small enough ensures that $\ti B_{\de_X}(\cW)\subset B_{\de_L}(\cW)$. Then $\ti d_\cW(u)=d_1(u)\sim|\la_1|$ on $\ta\in(\ta_0,\ta_X)$.
Hence, using the definition \eqref{def d0}-\eqref{def d1} of $d_1$ and the equations \eqref{Eqn:EstDynLambdaOne} of $\la_j$,
\EQ{
 \pt\p_\ta \ti d_\cW(u)^2 = \phi*\p_\ta 2\mu\la_1^2
  =\phi*[4\mu^2\la_1\la_2+O(\la_1^3)],
 \pr\p_\ta^2\ti d_\cW(u)^2
  = \phi*[4\mu^3(\la_1^2+\la_2^2)+O(\la_1^3)]+\phi'*O(\la_1^3)
  \sim \la_1^2 \sim \ti d_\cW(u)^2,}
where we also used Proposition \ref{Prop:extend v} to remove the convolution in the last step.
Since $\p_t\ti d_\cW(u(t_0))\ge 0$, the last estimate implies that $\ti d_\cW(u)$ is strictly increasing for $\ta\in(\ta_0,\ta_X)$.
It also implies exponential growth in $\ta$ of $\ti d_\cW$, so it is impossible to have the case $\ta_X=\I$ above.
In other words, there exists $T_X<T_+(u)$ such that $\ti d_\cW(u)$ reaches $\de_X$ at $\ta=\ta_X=\ta(T_X)$.
Since $\ti d_\cW\sim|\la_1|$ is positive continuous on $(\ta_0,\ta_X)$, $\la_1(\ta)$ cannot change the sign.
Let $\sg:=\sign\la_1(\ta)\in\{\pm\}$ be its sign.

Next we show the more precise exponential behavior.
Since $\p_\ta\ti d_\cW(u)^2\ge 0$ at $\ta=\ta_0$, there exists $\ta\in(\ta_0-2,\ta_0+2)$ where $\p_\ta\la_1^2\ge 0$, and so $\la_1\la_2\gec-|\la_1|^3$.
Since $\p_\ta(\la_1\la_2)\sim\la_1^2\sim \la_1(\ta_0)^2$ for $|\ta-\ta_0|<2$, there exists $\ta_1\in(\ta_0,\ta_0+2)$ such that $\la_1(\ta_1)\la_2(\ta_1)\gec-|\la_1(\ta_1)|^3$, or equivalently $\sg\la_2(\ta_1)\gec-\la_1(\ta_1)^2$.
Then $\sg\la_+(\ta_1)\ge|\la_1(\ta_1)|/2$ and $\sg\la_-(\ta_1)\ge 0$.
Let $\ti R:=|\la_1(\ta_1)|$, and suppose that for some $\ta_2\in(\ta_1,\ta_X)$
\EQ{
 \ta_1<\ta<\ta_2 \implies |\la_1(\ta)| \le 2\ti Re^{\mu(\ta-\ta_1)} \lec \de_X.}
Then the equations \eqref{Eqn:EstDynLambda} of $\la_\pm$ together with $\|v\|_{\dot H^1}\sim|\la_1|$ imply for $\ta\in(\ta_1,\ta_2)$,
\EQ{
 |\la_\pm(\ta)-e^{\pm\mu(\ta-\ta_1)}\la_\pm(\ta_1)| \lec \ti R^2 e^{2\mu(\ta-\ta_1)} \lec \de_X \ti Re^{\mu(\ta-\ta_1)},}
and so
\EQ{
 |\la_1|=\frac{\sg}{2}(\la_++\la_-) \CAS{ \le (1+C\de_X)\ti Re^{\mu(\ta-\ta_1)} <2\ti Re^{\mu(\ta-\ta_1)},\\
 \ge (1/4-C\de_X)\ti Re^{\mu(\ta-\ta_1)}> \ti Re^{\mu(\ta-\ta_1)}/5.}}
Hence the continuity in $\ta$ allows us to take $\ta_2=\ta_X$.
Moreover the above estimates together with $|\la_1|\sim \ti R$ on $(\ta_0,\ta_1)$ implies that, with $R:=\ti d_\cW(u(\ta_0))$,
\EQ{ \label{Eqn:Equivdlamb1}
 \ta_0\le\ta\le\ta_X \implies \ti d_\cW\sim\sg\la_1\sim Re^{\mu(\ta-\ta_0)}.}

In order to estimate $\ga$, consider the expansion of the energy \eqref{Eqn:ExpEnergy} without the $\gamma$ terms. We denote this expansion by $E_{\gamma^{\perp}}$:
\EQ{ \label{Eqn:ExpEnergyGamma}
 E_{\gamma^{\perp}}(u)  & := - \mu \lambda_{+} \lambda_{-} - C(\lambda_{+}  g_{+} + \lambda_{-}  g_{-})}
Notice that $ C'(f)= N(f)$. By this observation, \eqref{Eqn:DynLambdaplus} and \eqref{Eqn:DynLambdamin} we see that
\EQ{
 \partial_\ta E_{\gamma^{\perp}}(u)
& =  \LR{ N(v) - N (\lambda_{+} g_{+} + \lambda_{-} g_{-}),  g_{+}} \partial_\ta \lambda_{+} 
\prq +   \LR{  N(v) - N ( \lambda_{+} g_{+} + \lambda_{-}  g_{-}), g_{-}} \partial_\ta \lambda_{-}
\prq + \te_\ta  \left( \langle v, g_{-} \rangle \partial_\ta \lambda_{-}
+ \langle v , g_{+} \rangle \partial_\ta \lambda_{+}  \right)
\prq + \si_\ta  \left( \langle iv  , S'_{1} g_{-} \rangle  \partial_\ta \lambda_{-}
+ \langle iv, S'_{1} g_{+} \rangle \partial_\ta \lambda_{+} \right).}
This together with \eqref{Eqn:DynTheta} implies that, for $\ta\in[\ta_0,\ta_X]$,
\EQ{ \label{Eqn:EstDerEuEgamma}
| \partial_\ta ( E(u) - E(W) - E_{\gamma^{\perp}}(u))  |
  \lesssim  \lambda^{2}_{1} \| \gamma \|_{ \dot{H}^{1}} + \lambda^{4}_{1}.}
Moreover, using the elementary inequality  $| N(f)| \lesssim \max{(W^{3} |f|^{2}, |f|^{5})}$, Sobolev and H\"older, we see that
\EQ{
 \left| C(v) - C(\lambda_{+} g_{+} + \lambda_{-} g_{-})  \right| & \lesssim |\lambda_{1}|^{2} \| \gamma \|_{\dot{H}^{1}}.}
Now, by Proposition \ref{Prop:ControlOrth}
\EQ{  \label{Eqn:EstPointWiseNrj}
 E(u)- E(W) - E_{\gamma^{\perp}}(u) & = \frac{1}{2} \langle \mathcal{L} \gamma, \gamma \rangle +
C( \lambda_{+} g_{+} + \lambda_{-} g_{-}) - C(v) \\
& \sim \| \gamma \|^{2}_{\dot{H}^{1}} + O ( |\lambda_{1}|^{2} \| \gamma \|_{\dot{H}^{1}} ),}
which implies, by \eqref{Eqn:Equivdlamb1} and \eqref{Eqn:EstDerEuEgamma}, that \eqref{Eqn:Dynvperp} holds.
Plugging \eqref{Eqn:Dynvperp} into \eqref{Eqn:DynTheta} and then integrating in $\ta$, we obtain \eqref{Eqn:Dyntesi} as well.

Next we show \eqref{Eqn:DynK}. Expanding $K$, and using $L_+W=-4W^5$, we have
\EQ{ \label{Eqn:ExpandK2nd}
 K(W+v) & = - 4 \langle W^{5}, v \rangle + O (\| v \|^{2}_{\dot{H}^{1}}) \\
& = - 2 \mu \lambda_{1} \langle W,g_{2} \rangle - 4 \langle W^{5}, \gamma \rangle + O (\| v \|^{2}_{H^{1}}).}
This, combined with $\LR{W,g_2}>0$ (see \eqref{Eqn:CondOmega1}) as well as the above estimates on $\la_1,\ga,v$, proves \eqref{Eqn:DynK}.
\qedsymbol

\section{Sign functional} \label{Section:Sign}
In this section, we prove Proposition \ref{Prop:SignProp}.
On one hand, Proposition \ref{Prop:VarEst} implies that $\sign K$ is constant on each connected component of $\cH^\ep\setminus\ti B_{\delta}(\cW)$, provided that $\ep\le\ep_V(\de)$.
On the other hand, Proposition \ref{Prop:DistFuncEigen} implies that $\sign \lambda_{1}$ is constant on each connected component of $B_{\de_L}(\cW)\cap\ck\cH$.
Hence, after fixing $\de_1,\de_2,\ep_S$ such that $0<\de_1<\de_2\ll\de_X$ and $\ep_S\le\ep_V(\de_1)$, the functional $\Te$ is well-defined and continuous on $\ck\cH^{\ep_S}:=\cH^{\ep_S}\cap\ck\cH$ by \eqref{def Te}, once we prove that $-\sign \lambda_{1}= \sign K$ on
\EQ{
 Y:=\{\fy \in \ck\cH^{\ep_S} \mid \delta_{1} \leq \tilde{d}_{\cW} (\fy) \leq \delta_{2}\}.}
To this end, take any solution $u$ of \eqref{Eqn:SchrodCrit}  with initial data $u(0)\in Y$.
By applying Proposition \ref{Prop:DynEjection} either forward or backward in time, there exists $t_X\in I(u)$ such that $\ti d_\cW(u(t_X))=\de_X$ and $\ti d_\cW(u(t))\in[\ti d_\cW(u(0)),\de_X]\subset[\de_1,\de_X]$ between $t=0$ and $t=t_X$, so $u(t)$ remains in $\ck\cH^{\ep_S}\cap B_{\de_L}(\cW)\setminus \ti B_{\de_1}(\cW)$.
Hence $\sign\la_1(u(t))$ and $\sign K(u(t))$ are unchanged between $t=0$ and $t=t_X$, whereas $-\sign\la_1(u(t_X))=\sign K(u(t_X))$ by \eqref{Eqn:DynK} with $\de_2\ll\de_X$.
Therefore $-\sign\la_1=\sign K$ on $Y$.

Finally we prove that $\Te=\sign K$ on $\cH^0$.
Since $0\in\ck\cH$ is away from $\cW$, we have $\Te(0)=\sign K(0)=+1$.
If $E(u)<E(W)$ and $u\not=0$, then $K(u)\not=0$ by \eqref{Eqn:Charac2GdState}.
By
\EQ{
 \la\p_\la E(\la u)= K(\la u)=\la^2\|\na u\|_{L^2}^2-\la^6\|u\|_{L^6}^6,}
there is a unique $\la_0>0$ such that for $0<\la_-<\la_0<\la_+<\I$
\EQ{
 K(\la_-u)>0=K(\la_0u)>K(\la_+u).}
If $K(u)>0$, then $\{\la u\}_{0\le \la\le 1}$ is a $C^0$ curve in $\cH^0\subset\ck\cH$ connecting $u$ and $0$.
Hence by continuity $\Te(u)=+1$.
If $K(u)<0$, then $\{\la u\}_{1\le\la<\I}$ is a $C^0$ curve in $\cH^0$ connecting $u$ with the region $E(u)<0$, where $\cW$ is far and so $\Te=\sign K=-1$. Hence by continuity $\Te(u)=-1$.
The invariance of $\Te$ for the rotation and scaling follows from that of $K$ and $\la_1$, the latter being proved in Section \ref{Section:NonlinearDist}.
\qedsymbol

\section{One-pass lemma} \label{Section:OnePassLemma}
In this section we prove Proposition \ref{Prop:OnePassLemma}.

\subsection{Setting}
Let $\de,\ep>0$ and let $u$ be a solution satisfying $u(t_0)\in\cH^\ep\cap\ti B_\de(\cW)$ at some $t_0\in I(u)$.
The solution $u$ is fixed for the rest of proof, so we denote for brevity,
\EQ{
 \ti d(t):=\ti d_\cW(u(t)).}
We will define shortly the hyperbolic and the variational regions in $\cH^\ep$. In order to distinguish them, we use small parameters $\de_V,\de_M\in(0,\de_X]$, which will be fixed as absolute constants in the end.
First we impose the following upper bounds on $\de$ and $\ep$
\EQ{ \label{dep small}
 0<\de \ll \de_V \ll \de_M\le\de_X, \pq 0<\ep\le\min(\ep_S,\ep_V(\de_V)), \pq \ep<c_D\de.}
Since $\ep\le\ep_S$ and $\ep<c_D\de$, we have
\EQ{
 \cH^\ep\setminus\ti B_\de(\cW) \subset \cH^{\ep_S}\cap\ck\cH.}

Put $t_a:=\sup\{t_1\in(t_0,T_+(u))\mid t_0 < t\le t_1\implies \ti d(t)<\de\}$.
Since $\ti d(t_0)<\de$, we have $t_a\in(t_0,T_+(u)]$.
If $t_a=T_+(u)$ then \eqref{no return} holds with $t_+=T_+(u)$.
Hence we may assume without loss of generality that $t_a<T_+(u)$ and so $\ti d(t_a)=\de$.
If $ \{t\in(t_a,T_+(u))\mid \ti d(t)\le\de \}$ is empty, then \eqref{no return} holds with $t_{+} = t_{a}$.
If not, let $t_b:=\inf\{t\in(t_a,T_+(u))\mid \ti d(t)\le\de\}$. Applying Proposition \ref{Prop:DynEjection} at $t=t_a$ implies $t_a<t_b$.
Thus in the remaining case, we have
\EQ{
 t_0<\exists t_a<\exists t_b<T_+(u),
 \pq \ti d(t_a)=\de=\ti d(t_b)=\min_{t\in[t_a,t_b]}\ti d(t),}
from which we will derive a contradiction for small $\delta >0$ and for small $\epsilon > 0$ with $\delta-$dependent smallness. 

The rest of proof concentrates on the interval $[t_a,t_b]$, where
$u(t)$ stays in $\ck\cH$, and so $\Te(u(t))\in\{\pm 1\}$ is a constant, abbreviated by $\Te$ in the following.
$\ti d(t_a)=\de$ implies $E(u)=(1+O(\de^2))E(W)\gec 1$.

\subsection{Hyperbolic and variational regions}
Let $s$ be any local minimizer of the function $\ti d(t)$ on $[t_a,t_b]$ such that $\ti d(s) < \de_V$.
Then Proposition \ref{Prop:DynEjection} from $t=s$, forward and backward in time if $s \notin \{ t_{a},t_{b} \}$, forward in time if
$s = t_{a}$, and backward in time if $s=t_{b}$, yields a unique subinterval $I[s]\subset [t_a,t_b]$ such that
\begin{enumerate}
\item $\ti d(t)\sim -\Te \la_1(t) \sim \ti d(s)e^{\mu|\ta(t)-\ta(s)|}$ on $I[s]$,
\item $\ti d(t)^2$ is strictly convex as a function of $\ta$ on $I[s]$ with a minimum $<\de_V$ at $t=s$,
\item $\ti d(t)=\de_M$ on $\p I[s]\setminus\{t_a,t_b\}$.
\end{enumerate}

Let $\sL$ be the set of those local minimum points.
Since $\ti d(t)$ ranges over $[\de_V,\de_M]$ between any pair of points in $\sL$, its uniform continuity on $[t_a,t_b]$ implies that $\sL$ is a finite set.
Decompose the interval $[t_a,t_b]$ into the hyperbolic time $I_H$ and the variational time $I_V$ defined by the following
\EQ{
 I_H:=\bigcup_{s\in\sL} I[s], \pq I_V:=[t_a,t_b]\setminus I_H.}
By the definition of $\sL$ and $I[s]$, we have
\EQ{
 t\in I_H \implies \de\le\ti d(t)\le\de_M, \pq t\in I_V \implies \ti d(t)\ge\de_V.}
In particular, the coordinates $\si,\te,v,\la_{+},\la_{-},\la_{1},\la_{2},\ga$ and $\ta$ are defined on $I_H$.
Since $u$ is fixed, we regard those as functions of $t\in I_H$ in the rest of proof.

The soliton size on $I_H$ is measured by
\EQ{ \label{Eqn:EstSizeIH}
 m_{H} & :=  \sup_{t\in I_H} e^{-\si(t)} \sim \max_{s\in\sL} e^{-\si(s)},}
where the equivalence follows from \eqref{Eqn:Dyntesi} on each $I[s]$.
The hyperbolic dynamics in $\ta$ on $I[s]$ together with the time scaling $\dot\ta=e^{2\si}$ implies that
\EQ{
 e^{2\si(s)}|I[s]| \sim \log(\de_M/\ti d(s)) \in [\log(\de_M/\de_V),\log(\de_M/\de)].}

\subsection{Virial identity}
Now we consider a localized virial identity. For $m>0$, put
\EQ{
 \sV_m(t):= \LR{\phi_m u, ir\p_r u}.}
Then from the equation \eqref{Eqn:SchrodCrit}, we obtain
\EQ{ \label{Eqn:RHS}
 \dot\sV_m
 \pt = 2 \LR{|u_r|^2,\partial_{r}  r \phi_{m}}
 - \frac 12\LR{|u|^{2}, \De(r\p_r+3)\phi_m}
 - \frac{2}{3} \LR{|u|^6,(r\p_r+3)\phi_m}
 \pr= 2 K(\phi_{m} u) + 2 \int_{ m \leq|x| \leq 2m} \left(  |u_r|^{2}\partial_{r} r \phi_{m}
- | \partial_{r} (\phi_{m} u)|^{2} \right)  \, dx
\prQ - \frac{1}{2} \int_{ m \leq |x| \leq 2m} |u|^{2} \De(r\p_r+3)\phi_m dx
\prQ - \frac{2}{3} \int_{ m \leq |x| \leq 2m} (|u|^{6}(r\p_r+3)\phi_m  + 3|\phi_{m} u|^6) \, dx
\pr = 2 K(\phi_{m} u) + O(E_m),}
where
\EQ{
 E_m(t) := \int_{m \le |x|\le 2 m} |\na u|^{2} + |u/r|^2 + |u|^{6} \, dx.}
Using the decomposition \eqref{Eqn:Decompu}, \eqref{Eqn:VeVH1}, and Proposition \ref{Prop:extend v}, placing $W_\si $, $S_{-1}^{\si} v$ in $L^6_x$, and placing $\partial_{r}W_\si$, $\partial_{r} (S_{-1}^{\si}v)$ in $L^{2}_x$, we see that
\EQ{ \label{Eqn:EstBound}
 t\in\{t_a,t_b\}\implies |\sV_m(t)| \lec m^{2} \de.}

We need to estimate $K(\phi_{m} u)$ and $E_m$ on $I_{H}$ for $m$ to be chosen properly.
Using the scale invariance of $K$, \eqref{Eqn:ExpandK2nd}, and Sobolev's inequality as well, we obtain
\EQ{
 \pt K(\phi_m u)=K(\phi_{\ti m}(W+v))=K(W+(v-\phi_{\ti m}^C(W+v))
 \pr=-4\LR{W^5,v-\phi_{\ti m}^C(W+v)}+O(\|\phi_{\ti m}v-\phi_{\ti m}^CW\|_{\dot H^1}^2)
 \pr=-2\mu\la_1\LR{W,g_2}+O(\|\ga\|_{\dot H^1}+\|W^5\|_{L^{6/5}(|x|>\ti m)}+\|\phi_{\ti m}^C W\|_{\dot H^1}^2+\|v\|_{\dot H^1}^2),}
where $\ti m(t):=me^{\si(t)}$.
The decay of $W$ implies
$\|W^5\|_{L^{6/5}(|x|>\ti m)} \lec \LR{\ti m}^{-5/2}$
and $\|\phi_{\ti m}^C W\|_{\dot H^1}^2 \lec \LR{\ti m}^{-1}$.
Plugging these into the above yields, for $s\in\sL$ and $t\in I[s]$,
\EQ{ \label{Eqn:ExpKa}
 K(\phi_m u) = -2\mu\LR{W,g_2}\la_1 + O(\ti d(s)+\la_1^2+\LR{\ti m}^{-1}).}
Similarly, the decomposition \eqref{Eqn:CondOmega1} and Hardy inequality yield for $t\in I_H$
\EQ{ \label{Eqn:EstXi}
 \int_{|x|>m}( |\na u|^2+|u/r|^2+|u|^6 )dx  \lec  \la_1^2 + \LR{\ti m}^{-1}.}
Putting these estimates into \eqref{Eqn:RHS} yields, on each $I[s]$,
\EQ{
 \dot \sV_m = -4\mu\LR{W,g_2}\la_1 + O(\ti d(s)+\la_1^2+\LR{\ti m}^{-1}).}
Then using the hyperbolic dynamics of $\la_1$ and $\dot\ta=e^{2\si}\sim e^{2\si(s)}$, we obtain, with some absolute constant $C_H\ge 1$,
\EQ{
 [\Te \sV_m]_{\p I[s]} \pt\gec e^{-2\si(s)}[\de_M-C_H\LR{\ti m(s)}^{-1}\log(\de_M/\de)]
 \pr\ge e^{-2\si(s)}[\de_M-C_H(m_H/m)\log(\de_M/\de)].}
Thus we obtain
\EQ{ \label{m for H}
 m > m_X := 2C_H m_H\de_M^{-1}\log(\de_M/\de) \implies \int_{I_H}\Te\dot\sV_m dt \gec m_H^2\de_M.}

In order to control the cut-off error in $I_V$, we introduce
\EQ{ \label{def cIV}
 \mathcal{I}_{V} := \int_{I_{V}} \int_{\R^3}  |\nabla u|^{2} + |u/r|^2 + |u|^{6} \, dx \, dt
 \sim \int_{I_{V}} \int_{\R^3} |\nabla u|^{2} \, dx \, dt,}
where the equivalence follows from Hardy's inequality and $E(u)\sim E(W)>0$.
Since
\EQ{
 \int_0^\I  \int_{I_{V}} \int_{|x|>m} \frac{m}{r}(|\nabla u|^{2}  + |u/r|^2+ |u|^{6}) dx dt\frac{dm}{m}
 = \mathcal{I}_{V},}
there exists $m\in(m_0,m_1)$ for any $m_1>m_0>0$ such that
\EQ{ \label{Eqn:ObsBlow}
 \int_{I_{V}} \int_{|x|>m} \frac{m}{r} (|\nabla u|^{2}  + |u/r|^2+ |u|^{6})dx \, dt \le \frac{\cI_V}{\log(m_1/m_0)}.}

\subsection{Blow-up region} \label{Subsection:OnePassLemmaBlow}
We start with the simpler case $\Te(u)=-1$, where the solution will blow up.
First consider $\dot\sV_m$ on $I_{V}$.
Since $\ep<\ep_V(\de_V)$ and $\ti d>\de_V$ on $I_V$, Proposition \ref{Prop:VarEst} implies
\EQ{ \label{Eqn:LowerBdKBlow}
 t\in I_V \implies -K(u) \ge \ka(\de_V) \ge c_0(\de_V)\|\na u\|_{L^2}^2,}
for some constant $c_0(\de_V)>0$, since for $\|\na u\|_{L^2}^2\ge 4E(W)$ we have
\EQ{
 -K(u)=2\|\na u\|_{L^2}^2-6E(u) > 2\|\na u\|_{L^2}^2-6(E(W)+O(\de^2)) \ge \frac{\|\na u\|_{L^2}^2}{4}.}

In order to estimate $\dot\sV_m$ with $K(u)$ on $I_V$, the optimal cut-off radius is given by
\EQ{
 m_V^-:=\sup\Bigl\{R>0 \Bigm| \int_{I_V}\int_{|x|>R}(|\na u|^2-|u|^6)dxdt\le 0\Bigr\}.}
$K(u)<0$ on $I_V$ implies $m_V^->0$, while the Sobolev inequality on $|x|>R$ implies that $m_V^-<\I$.
For any $m \geq m_V^-$, we have from \eqref{Eqn:RHS} and \eqref{Eqn:LowerBdKBlow},
\EQ{ \label{Eqn:VirialIV}
 t\in I_V \implies -\dot\sV_m \pt= -2\int_{|x|\le m}(|\na u|^2-|u|^6)dx + O(E_m)  \pr\ge -2K(u)+O(E_m) \ge c_0(\de_V)\|\na u\|_{L^2}^2 + O(E_m).}
The last term can be absorbed by the other, using \eqref{Eqn:ObsBlow}.
Hence for any $m_0\ge m_V^-$, there exists $m\in(m_0,C_1(\de_V)m_0)$ for some constant $C_1(\de_V)>1$ such that
\EQ{ \label{m for V}
 \int_{I_V}-\dot \sV_m dt \gec c_0(\de_V)\cI_V.}
The minimal $m>0$ satisfying this and \eqref{m for H} satisfies
\EQ{
 m \le C_1(\de_V)\max(m_V^-,m_X),}
for which \eqref{Eqn:EstBound} with $[t_a,t_b]=I_H\cup I_V$ implies
\EQ{ \label{vir all-}
 m_H^2\de_M + c_0(\de_V)\cI_V \lec [-\sV_m]_{t_a}^{t_b} \pt\lec m^2\de
 \pr\lec (C_1(\de_V)m_V^-)^2\de + (C_1(\de_V)m_X)^2\de.}

Now we impose an upper bound on $\de$ by the condition
\EQ{ \label{mX small}
 (C_1(\de_V)m_X)^2\de \ll m_H^2\de_M,}
which is equivalent to
\EQ{ \label{bd deB-1}
 (C_1(\de_V)C_H)^2\de\log^2(\de_M/\de) \ll \de_M^3.}
Then the last term in \eqref{vir all-} is absorbed by the first one, hence
\EQ{ \label{vir red-}
 m_H^2 \de_M + c_0(\de_V)\cI_V \lec (m_V^- C_1(\de_V))^2\de .}

In order to bound $m_V^-$, we use the equation for $|u|^2$:
\EQ{ \label{Eqn:IdentLocMass}
 \p_t |u|^{2} & =  2 \Im \left( \nabla \cdot (\nabla u \bar{u}) \right).}
Multiplying it with $\phi_{m/2}^C/r^2$ and integrating on any interval $J\subset I_V$, we obtain
\EQ{ \label{Eqn:DecayEstVar}
 [\LR{|u/r|^2,\phi_{m/2}^C}]_{\partial J}
 \pt\lec
 \int_{J} \int_{|x|>m/2}\frac{|uu_r|}{r^3} dxdt
 \pr\lec \frac{1}{|\mathcal{I}_{V}|} \int_{I_V} \int_{|x|>m/2} \frac{m^2}{r^2}( |u_r|^{2} + |u/r|^2 ) dx dt,}
for $m\ge \cI_V^{1/2}$.
On the other hand, we have from \eqref{Eqn:EstXi},
\EQ{
 t\in \p I_V \implies \int_{|x|>m/2}|u/r|^{2} dx
  \lec m_H/m+\de_M^2.}
Hence by \eqref{Eqn:ObsBlow}, there exists $m\sim\max(m_H,\cI_V^{1/2})$ such that
\EQ{ \label{ext u/r small}
 \sup_ { t \in I_{V} }\|u/r\|_{ L^{2} (|x|>m ) } \ll 1.}

Using the radial Gagliardo-Nirenberg inequality
\EQ{ \label{Eqn:RadSobHardyIneq}
 \| r^{1/2}\fy\|_{L^{\infty} (|x|> m)} & \lesssim \|\p_r \fy \|_{L^{2} (|x|> m) }^{1/2}  \|\fy/r\|_{L^{2} (|x|> m)}^{1/2},}
we have, for any $m\ge 0$ and $\fy\in\dot H^1\rad$,
\EQ{ \label{ext GN}
 \int_{|x|>m}|\fy|^6dx \pt\le \|r^{1/2}\fy\|_{L^\I(|x|>m)}^4\|\fy/r\|_{L^2(|x|>m)}^2
 \pr\lec \|\na\fy\|_{L^2(|x|>m)}^2\|\fy/r\|_{L^2(|x|>m)}^4.}
Plugging \eqref{ext u/r small}, we obtain for $t\in I_V$,
\EQ{
 \|u\|_{L^6(|x|>m)}^6 \lec \|\na u\|_{L^2(|x|>m)}^2\|u/r\|_{L^2(|x|>m)}^4 \ll \| \nabla u \|^{2}_{L^{2} (|x|> m)},}
which implies $\int_{|x|>m}( |\na u|^2-|u|^6 )dx\ge 0$,
and, by the definition of $m_V^-$,
\EQ{ \label{bd mV-}
 m_V^- \le m \sim \max(m_H,\cI_V^{1/2}).}
Hence, imposing another upper bound on $\de$ by the condition:
\EQ{ \label{bd deB-2}
 C_1^2(\de_V)\de \ll c_0(\de_V),}
we see that \eqref{bd mV-} contradicts \eqref{vir red-}.

In conclusion, the smallness conditions on $\de,\ep$ in the case $\Te=-1$ are \eqref{bd deB-1}, \eqref{bd deB-2}, and \eqref{dep small}, which determine $\de_B$ and $\ep_B$.

\subsection{Scattering region} \label{Subsection:OnePassLemmaScatt}
Now we consider the case $\Te(u)=+1$, where the solution will scatter.
The argument is similar to that in the previous case, but more involved.
In particular, we need several smallness conditions on $\de_M$.

First observe that there exists an absolute constant $C_E  \sim 1$ such that
\EQ{ \label{Eqn:ControlOfH1}
  1/C_E \le \| u(t) \|^{2}_{\dot{H}^{1}} \le C_E }
for all $t\in I(u)$. Indeed, since $E(u)\sim E(W)\sim 1$, the upper bound follows from \eqref{Eqn:RelNrjK}, while the lower bound follows from $E(\fy)\sim\|\fy\|_{\dot H^1}^2$ for small $\fy\in\dot H^1$.

Next we estimate $K(\phi_{m}u)$ on $I_V$.
Using \eqref{Eqn:DefI} and Proposition \ref{Prop:VarEst}, we see that
\EQ{
 I( \phi_m u)  \le I(u) < E(W)+\ep^2-\ti\ka(\de_V)/2 \le I(W) - \ti\ka(\de_V)/3,}
provided that
\EQ{ \label{bd ep+}
 6\ep^2 \le \ti\ka(\de_V):=\min(\ka(\de_V), c_V/C_E).}
Then using \eqref{Eqn:SobIneqMax} and \eqref{Eqn:StatSolW}, we obtain
\EQ{ \label{Eqn:LowerBdKVarScat1}
 K (\phi_m u) & \geq \| \nabla( \phi_m u ) \|^{2}_{L^{2}}  \left( 1 - \|W\|_{L^6}^2\| \phi_m u \|^{4}_{L^{6}}/\|W\|_{\dot H^1}^2  \right) \\
 & \geq  \| \nabla( \phi_m u ) \|^{2}_{L^{2} }\left[1- \|W\|_{L^6}^2(\|W\|_{L^6}^6-\ti\ka)^{4/6}/\|W\|_{\dot H^1}^2  \right]
 \pr\gec \ti\ka(\de_V)\| \nabla (\phi_m u )\|^{2}_{L^{2}}
 \ge \ti\ka(\de_V)\|\na u\|_{L^2(|x|<m)}^2. }

In order to decide the cut-off for $I_V$, we put
\EQ{ \label{Eqn:DefmV}
 m_V^+(\de_M):=\inf\Bigl\{R>0 \Bigm| \int_{I_{V}} \int_{|x| < R} |\nabla u|^{2} \, dx \, dt & \geq  \de_M^3|I_{V}|\Bigr\}\in(0,\I).}
By \eqref{Eqn:RHS}, \eqref{Eqn:ObsBlow}, \eqref{Eqn:ControlOfH1}, and \eqref{Eqn:LowerBdKVarScat1}, there exists $C_2(\de_V)\in(1,\I)$ such that for any $m_0\ge m_V^+(\de_M)$, there exists $m\in(m_0,C_2(\de_V)m_0)$ such that
\EQ{ \label{m for V+}
 \int_{I_V} \dot\sV_m dt \gec \ti\ka(\de_V)\de_M^3|I_V|.}
Note that $\de_V$ depends on $\de_M$ through the condition $\de_V\ll\de_M$ in \eqref{dep small}, which allows us to determine $C_2$ in terms of $\de_V$ only. 
The minimal $m>0$ satisfying both this and \eqref{m for H} must satisfy
\EQ{
 m \le C_2(\de_V)\max(m_V^+(\de_M),m_X),}
for which \eqref{Eqn:EstBound} implies
\EQ{  \label{vir all+}
 m_H^2\de_M + \ti\ka(\de_V)\de_M^3|I_V| \sim [\sV_m]_{t_a}^{t_b} \pt\lec m^2\de
 \pr \le (C_2(\de_V)m_V^+(\de))^2\de + (C_2(\de_V)m_X)^2\de.}

Now we impose an upper bound on $\de$ by the condition
\EQ{
 (C_2(\de_V)m_X)^2\de \ll m_H^2\de_M,}
which is equivalent to
\EQ{ \label{bd deB+1}
  (C_2(\de_V)C_H)^2\de\log^2(\de_M/\de) \ll \de_M^3.}
Then the last term in \eqref{vir all+} is absorbed by the first one, and
\EQ{ \label{vir red+}
 m_H^2\de_M + \ti\ka(\de_V)\de_M^3|I_V| \pt\lec (C_2(\de_V)m_V^+(\de))^2\de
 \pr\ll (m_V^+(\de_M))^2\de_M^3\log^{-2}(\de_M/\de).}
Imposing another upper bound on $\de$ by
\EQ{ \label{bd deB+2}
  \ti\ka(\de_V)^{1/2}\log(\de_M/\de) \gg e^{1/\de_M^3},}
yields
\EQ{ \label{mV super}
 m_H + |I_V|^{1/2} \ll e^{-1/\de_M^3}m_V^+(\de_M).}
To compare $m_H$ with $|I_V|^{1/2}$,
use \eqref{Eqn:ObsBlow} with $m_0:=m_H+|I_V|^{1/2}$ and $m_1:=m_V^+(\de_M)/2$. Then there exists $m\in(m_0,m_1)$ such that
\EQ{ \label{def m}
  \int_{I_V}\int_{|x|>m}\frac{m}{r}(|\na u|^2+|u|^6+|u/r|^2)dxdt \lec \de_M^3|I_V|,}
because of \eqref{mV super}.
Using Hardy, we have
\EQ{
 \|u/r\|_{L^2(|x|<2m)}
 \pt\lec \|\phi_m u\|_{\dot H^1}+\|u/r\|_{L^2(m<|x|<2m)}
 \pr\lec \|u/r\|_{L^2(m<|x|<2m)}+\|\na u\|_{L^2(|x|<2m)}.}
Integrating its square over $I_V$, and using the definition of $m_V^+>2m$ for the $\| \nabla u \|^{2}_{L^{2}(|x| < 2m)}$ term, and
\eqref{def m} for the the $\|u/r\|^{2}_{L^{2}( m <  |x| < 2m)}$ term, we obtain
\EQ{
 \int_{I_V}\int_{|x|<2m}[|u/r|^2+|\na u|^2]dxdt \lec \de_M^3|I_V|.}
Using the radial Sobolev inequality, we have for any $\fy\in\dot H^1\rad$ and $m>0$,
\EQ{ \label{u/r 2 u6}
 \|\fy\|_{L^6(|x|<m)}^6 \lec \|\fy/r\|_{L^2(|x|<m)}^2\|r^{1/2}\fy\|_{L^\I}^4
 \lec \|\fy/r\|_{L^2(|x|<m)}^2\|\fy_r\|_{L^2}^4.}
Inserting it to the above estimate yields
\EQ{ \label{int IV small}
 \int_{I_V}\int_{|x|<2m}[|\na u|^2+|u/r|^2+|u|^6]dxdt \lec \de_M^3|I_V|.}
Noting that $m_X\ll m_H<m$ by \eqref{bd deB+2} and $m_H<m_0$, we see from the above estimate that if $m_H^2\gg\de_M^2|I_V|$, then by \eqref{Eqn:RHS} and \eqref{m for H}
\EQ{
 \int_{I_V} |\dot\sV_{m_H}| dt \ll m_H^2\de_M \lec [\sV_{m_H}]_{t_a}^{t_b} \lec m_H^2\de,}
which contradicts $\de\ll\de_M$. Therefore
\EQ{ \label{mH IV}
 m_H^2 \lec \de_M^2|I_V|.}

Next we estimate $\|u/r\|_{L^2(I_V\times\R^3)}$.
By the same argument as for \eqref{Eqn:DecayEstVar}, we have for any interval $J\subset I_V$,
\EQ{
 \left[\LR{|u/r|^2,\phi_m^C}\right]_{\partial J}
  \pt\lec \int_J \int_{|x|>m}\frac{| uu_r |}{r^3} \, dx \, dt
 \pr\lec \frac{1}{|I_{V}|} \int_{I_{V}} \int_{|x|>m} \frac{m^{2}}{r^2} ( |u_r|^{2} + |u/r|^{2})dxdt \lec \de_M^3,}
where we used $m>m_0>|I_V|^{1/2}$ and \eqref{def m}.
On the other hand, we have from \eqref{Eqn:EstXi} and then \eqref{mH IV},
\EQ{
 t\in\p I_V \implies \int_{|x|>m}|u/r|^{2} dx
  \lec \frac{m_H}{m}+\de_M^2 < \frac{m_H}{|I_V|^{1/2}} + \de_M^2 \lec \de_M.}
Combining the above two estimates yields
\EQ{ \label{Eqn:EstVar}
 \sup_{t \in I_{V}} \|u/r\|_{L^{2}(|x|>2m)}^2 \lec \de_M.}
Then using \eqref{int IV small} for the integral over $|x|<2m$, we obtain
\EQ{
 \int_{I_V}\int_{\R^3}|u/r|^2dxdt \lec \de_M|I_V|.}
Decomposing $I_V$ into its connected components, we obtain an interval $I\subset I_V$ such that $\p I\subset \p I_V$ and
\EQ{ \label{u/r small}
 \int_I \int_{\R^3}|u/r|^2dxdt \lec \de_M|I|.}

Now we resort to an argument by Bourgain \cite{bourgjams}, in order to reduce the problem to energy below the ground state $E(W)$.
Although Bourgain in \cite{bourgjams} treated the defocusing case, the perturbative argument works as well for the focusing equation \eqref{Eqn:SchrodCrit} under the uniform bound \eqref{Eqn:ControlOfH1} in $\dot H^1$, while the non-perturbative argument with the Morawetz estimate can be replaced with \eqref{u/r small}, as is shown below.

In order to apply the argument to the interval $I$, the first observation is
\EQ{ \label{low bd StIV}
 \|u\|_{S(I)}\gec 1.}
\begin{proof}
Let $t_0:=\inf I\in\p I\subset \p I_V$.
Then by the definition of $I_V$, we have $\ti d(t_0)=\de_M$ and Proposition \ref{Prop:DynEjection} from $t=t_0$ yields some $t_1\in I_V$ such that $\ti d(t_1)=\de_X$ and $\p_t\ti d(t)>0$ on $(t_0,t_1)\subset I$.
It suffices to show $\|u\|_{S(t_0,t_1)}\gec 1$.
The scaling invariance reduces it to the case $\si(t_0)=0$.
Then $\de_M\ll\de_X$ and $\dot\ta=e^{2\si}$ with \eqref{Eqn:Dyntesi} and \eqref{Eqn:Dyndq} imply that $t_1>t_0+1$.
Put $v:=u-W$. Then from the equation
\EQ{
 i\dot v-\De v = 5W^4v_1+iW^4v_2+N(v),}
the embedding $W^1\subset S$, and the Strichartz estimate \eqref{Strich}, we obtain for any interval $J=[a,b]\subset (t_0,t_0 + 1)$,
\EQ{
 \|v\|_{(W^1\cap L^\I_t\dot H^1)(J)}
 \pt\lec  \|v(a)\|_{\dot H^1}+ [\|W\|_{W^1(J)}+\|v\|_{W^1(J)}]^4\|v\|_{W^1(J)}
 \pr\lec \|v(a)\|_{\dot H^1}+ |J|^{2/5}\|v\|_{W^1(J)}+\|v\|_{W^1(J)}^5.}
Hence if $|J| \ll 1$ then
\EQ{
 \|v\|_{(W^1\cap L^\I_t\dot H^1)(J)}
 \lec  \|v(a)\|_{\dot H^1}.}
Repeating this estimate from $t=t_0$ on consecutive small intervals, we obtain
\EQ{
 \|v\|_{S(t_0,t_0+1)}\lec \|v\|_{W^1(t_0,t_0+1)} \lec \|v(t_0)\|_{\dot H^1} \lec \de_M,}
so,
$\|u\|_{S(t_0,t_0+1)}\ge \|W\|_{L^{10}_x}-O(\de_M) \gec 1$.
\end{proof}

Hence as in \cite{bourgjams}, we can decompose the interval $I$ such that
\EQ{
 I=[t_0,t_N],\pq t_0<t_1<\cdots<t_N,\pq I_j:=[t_j,t_{j+1}],\pq \|u\|_{S(I_j)}\in[\y,2\y)}
for a small fixed constant $\y>0$.
In the following, $c$ denotes a small positive constant, and $C(\y)$ denotes a large positive constant which may depend on $\y$, both allowed to change from line to line.

By the perturbation argument from Section 3 to (4.11) in \cite{bourgjams}, where the sign of nonlinearity is irrelevant, we have for each $j$,
\EQ{
 \|u\|_{\bar{S}^1(I_j)} \lec 1,}
and there exist a subinterval $I'_j\subset I_j$ and $R_j\lec |I'_j|^{1/2}$ such that
\EQ{
 \pt \inf_{t\in I'_j}\min(\|\na u(t)\|_{L^2(|x|<C(\y)R_j)},\|u(t)\|_{L^6(|x|<C(\y)R_j)}) \gec \y^{3/2}.}
Combining it with \eqref{u/r 2 u6} and \eqref{u/r small} yields
\EQ{
 \sum_{j=1}^N |I'_j| \le C(\y)\de_M\sum_{j=1}^N|I_j|.}
Hence there exists $j\in\{1,\dots,N\}$ such that
\EQ{ \label{conc Ij}
 R_j^2 \lec |I'_j| \le C(\y)\de_M|I_j|.}
Fix $s\in I'_j$.
By the time reversal symmetry, we may assume without loss of generality
\EQ{ \label{forward}
 t_{j+1}-s>s-t_j.}
By \cite[Lemma 5.12]{bourgjams}, there exists $R\le C(\y)R_j$ such that
\EQ{ \label{Eqn:EstH1w}
 \|\phi_R^C u(s)\|_{\dot H^1}^2 < \|u(s)\|_{\dot H^1}^2 - c\y^3.}
Let $v$ be the solution of the free Schr\"odinger equation with initial data
\EQ{
 v(s):= \phi_R u(s),}
and $w:=u-v$.
By the $L^p$ decay estimate \eqref{Lp decay}, H\"older's inequality (placing $u$ in $L^{6}$ and $\nabla u$ in $L^{2}$), and \eqref{Eqn:ControlOfH1}, we see that
\EQ{ \label{Eqn:Dispersivev1}
 \pt \| v (t) \|_{L^{6}_x} \lesssim |t-s|^{-1} \| v(s) \|_{L^{6/5}_x} \lesssim R^{2}|t-s|^{-1},
 \pr \| \nabla v(t) \|_{L^{30/13}_x} \lec |t-s|^{-1/5} \| \nabla v(s) \|_{L^{30/17}_x} \lec R^{1/5} |t - s|^{-1/5}.}
Hence using \eqref{conc Ij} and \eqref{forward}, we obtain
\EQ{ \label{vdec}
 \pt \|v(t_{j+1})\|_{L^6_x} \lec R^2/|t_{j+1}-s| \le C(\y)\de_M
 \pr \|v\|_{W^1(t_{j+1},\I)} \lec (R^2/|t_{j+1}-s|)^{1/10} \lec C(\y)\de_M^{1/10}.}

By the equation, integration by part, and H\"older's inequality, we have
\EQ{ \label{Eqn:ControlL6}
 \left[\| u(t) \|^{6}_{L^{6}}\right]_{s}^{t_{j+1}}
 \pt= 6 \int_{s}^{t_{j+1}}\Re\LR{|u|^{4}u,\dot u}dt
 \pr= 6 \int_{s}^{t_{j+1}}\Im\LR{\na(|u|^{4}u),\na u}dt
 \lec \|u \|^{2}_{\bar S^1(I_j)} \| u \|^{4}_{S(I_j)} \lec \eta^{4}.}
By the Strichartz estimate \eqref{Strich}, we have
\EQ{ \label{Eqn:ControlH1w}
 \| w( t_{j+1} ) \|_{\dot{H}^{1}}- \| w(s) \|_{\dot{H}^{1}} \lec \|u \|_{W^1(I_{j})} \| u \|^{4}_{S(I_{j})}
 \lec \eta^{4}.}
Let $\tilde{w}$ be the solution of \eqref{Eqn:SchrodCrit} with initial data $\tilde{w}(t_{j+1}):= w(t_{j+1} )$.
Then by the above estimates together with \eqref{Eqn:EstH1w}, \eqref{vdec}, and \eqref{Eqn:ControlL6}, we get
\EQ{ \label{Etiw dec}
 E(\ti w) & \leq
\frac{1}{2} \| \nabla w ( t_{j+1} ) \|^{2}_{L^{2}} -  \frac{1}{6} \| u ( t_{j+1} )  \|^{6}_{L^{6}} + O(\eta^{4})
\\ & \leq \frac{1}{2} \| \nabla w (s) \|^{2}_{L^{2}} - \frac{1}{6} \| u(s) \|^{6}_{L^{6}} + O(\eta^{4})
 \pr\leq E(u) -  c \eta^{3} \le E(W) + \ep^2 - c\y^3 < E(W)-c\y^3/2,}
where in the first and last steps, we imposed upper bounds on $\de_M$ and $\epsilon$ respectively: 
\EQ{\label{epseta}
 C(\y)\de_M\ll\y^4, \pq \epsilon \ll \y^{3/2}.}

Similarly, plugging \eqref{Eqn:EstH1w} into \eqref{Eqn:ControlH1w} yields
\EQ{ \label{nablauinfnablaW}
\| \nabla  \tilde{w} (t_{j+1}) \|^{2}_{L^{2}}  \leq \| \nabla u (s) \|^{2}_{L^{2}} - c \eta^{3},}
while $E(u) < E(W) + \epsilon^{2}$ and \eqref{Eqn:RelNrjK} together with $K(u(s))>0$ implies
\EQ{
  \| \nabla u(s) \|^{2}_{L^{2}} <  \| \nabla W \|^{2}_{L^{2}} + 3\epsilon^{2}}
Hence
$
\| \nabla \tilde{w} (t_{j+1}) \|^{2}_{L^{2}} < \| \nabla W \|^{2}_{L^{2}}
$
and therefore by \eqref{Eqn:Charac3GdState},
\EQ{
 K(\ti w(t_{j+1}))>0.}

Hence by the result of Kenig and Merle \cite{kenmer} below the ground state energy, \eqref{Eqn:ControlOfH1} and \eqref{nablauinfnablaW}, $\ti w$ scatters in both time directions with a uniform Strichartz bound:
\EQ{
 \|\ti w\|_{W^1(\R)} < C(\y).}
In order to control $u$ by this, we use the long-time perturbation \cite[Theorem 2.14]{kenmer}:
\begin{lem}[\cite{kenmer}] \label{res:perturbation}
Let $u$ be a solution of \eqref{Eqn:SchrodCrit}. Let $I$ be an interval with some $t_0\in I\cap I(u)$. Let $e\in N^1(I)$ and let $\ti u\in C(I;\dot H^1)$ be a solution of
\EQ{
i \p_{t} \ti u - \triangle \ti u = |\ti u|^{4} \ti u + e.
}
Assume that for some $B_1,B_2,B_3>0$
\EQ{
 \|\ti u\|_{L^\I_t \dot{H}^{1}(I)}  \leq B_1,
 \pq \| \tilde{u} \|_{S(I)}  \leq B_2,
 \pq \| \tilde{u}(t_{0}) - u(t_{0}) \|_{\dot{H}^{1}}  \leq B_3.}
Then there exists $\nu_P=\nu_P(B_1,B_2,B_3) >0$ such that if
\EQ{
 \| e^{-i (t -t_{0})  \De} (\tilde{u}(t_{0}) - u(t_{0})) \|_{S(I)}
 +\| e \|_{N^1(I)} =:\nu\le\nu_P,}
then $ I \subset I(u)$ and
\EQ{ \label{Eqn:BoundPrecSI}
  \|\tilde{u} \|_{S(I)}  \lec 1,
  \pq \| \tilde{u} - u \|_{L_{t}^{\infty} \dot{H}^{1} (I)} \lec \nu+B_3,}
where the implicit constants depend on $B_1,B_2,B_3$.
\end{lem}
Apply the above lemma to $u$ and $\ti u:=\ti w+v$ with $I=[t_{j+1},\I)$ and initial data at $t=t_{j+1}$.
From the bounds on $\ti w$ and $v$, we have
\EQ{ \label{bd tiu}
 \pt \|\ti u\|_{L^\I_t\dot H^1} \lec 1,
 \pq \|\ti u\|_{S} \le C(\y),
 \pq \ti u(t_{j+1})-u(t_{j+1})=0,}
and, using \eqref{vdec}, there exists a large positive constant $C_*(\y)$ such that
\EQ{
 \pn\| e \|_{N^1(I)}
 \pt=\||\ti w|^4\ti w-|\ti u|^4\ti u\|_{N^1(I)}
 \pr\lec (\|\ti w\|_{W^1(I)}+\|v\|_{W^1(I)})^4\|v\|_{W^1(I)}
 \le C_*(\y)\de_M^{1/10}.}
So by imposing another smallness condition on $\de_M$:
\EQ{ \label{cond deM2}
 C_*(\y)\de_M^{1/10}\ll \nu_P(C(\y),C(\y),0),}
we can apply the above lemma. Hence there exists another large positive constant $C_{**}(\y)$ such that
\EQ{ \label{approx tiu}
 \|\ti u-u\|_{L^\I_t\dot H^1_x(t_{j+1},\I)}\le C_{**}(\y)\de_M^{1/10}.}

Since $K(\ti w)>0$, which is preserved in time because of $E(\ti w)<E(W)$,
we have, using \eqref{Eqn:DefI} and \eqref{Etiw dec},
\EQ{ \label{Eqn:Estw2GroundState}
 \| \ti w(t) \|^{6}_{L^6_x} & \leq \| W \|^{6}_{L^6_x}  - \frac{3c}{2} \eta^{3}.}
Taking $\de_M$ smaller if necessary we have
\EQ{ \label{cond deM3}
 C_{**}(\y)\de_M^{1/10} \ll \y^3.}
Hence, combining the above estimates with \eqref{Eqn:Dispersivev1} and \eqref{vdec}, and taking $\de_{M}$ smaller if necessary, we obtain
\EQ{
 \|u(t)\|_{L^6_x}^6 \le \|W\|_{L^6_x}^6 - c\y^3,}
which contradicts $\tilde{d}_{\cW} (u(t_{b})) = \de\ll\de_M$, since \eqref{cond deM3} implies $\de_M \ll \y^3$.
In conclusion, after fixing the constant $\de_M>0$ such that \eqref{cond deM2}, \eqref{epseta} and \eqref{cond deM3} hold, the smallness conditions on $\de,\ep$ in the case $\Te=+1$ are \eqref{dep small}, \eqref{bd deB+1}, \eqref{bd deB+2}, and \eqref{epseta}, which determine $\de_B$ and $\ep_B$.
\qedsymbol

\section{Solutions staying around the ground states} \label{Sect:staying}
In this section, we prove Proposition \ref{Prop:staying}.
Let $u$ be a solution of \eqref{Eqn:SchrodCrit} satisfying $u(t_0)\in\cH^{\ep_B(\de)}\cap \ti B_\de(\cW)$ for some $\de\in(0,\de_B]$ and $t_0\in I(u)$, and $t_+=T_+(u)$, namely $\ti d_\cW(u(t))<\de$ for $t\in[t_0,T_+(u))$.

If $u(t)\in\ck\cH$ and $\p_t\ti d_\cW(u(t))\ge 0$ at some $t\in[t_0,T_+(u))$, then Proposition \ref{Prop:DynEjection} implies that $\ti d_\cW(u(t))$ increases up to $\de_X>\de_B>\de$, contradicting $t_+=T_+(u)$.
Hence for all $t\in[t_0,T_+(u))$,
\EQ{
 c_D\ti d_\cW(u(t))\le\sqrt{E(u)-E(W)}\pq\text{or}\pq \p_t\ti d_\cW(u(t))<0,}
so by the mean value theorem, there are only two possibilities:
\begin{enumerate}
\item There exists $t_1\in[t_0,T_+(u))$ such that $u(t)\not\in\ck\cH$ for all $t\in[t_1,T_+(u))$.
\item $u(t)\in\ck\cH$ and $\p_t\ti d_\cW(u(t))<0$ for all $t\in[t_0,T_+(u))$.
\end{enumerate}
In the first case, if we choose the minimal $t_1$, then for $t_0\le t<t_1$, we have $\p_t\ti d_\cW(u(t))<0$.
Hence it suffices to treat the latter case, for which $t_1=T_+(u)$.
Since $\ti B_\de (\cW) \subset \ti B_{\de_X} (\cW) \subset B_{\de_L} (\cW)$, Proposition \ref{Prop:extend v} implies that $\ta\to\I$ as $t\nearrow T_+(u)$.
Apply Proposition \ref{Prop:DynEjection} backward in time from any $t\in(t_0,T_+(u))$, corresponding to $\ta\in(\ta(t_0),\I)$. Then
\EQ{
 \ti d_\cW(u(t_0)) \sim e^{\mu(\ta(t))-\ta(t_0))}\ti d_\cW(u(t)).}
Sending $t\nearrow T_+(u)$ yields $\ti d_\cW(u(t))\to 0$.
\qedsymbol

\section{Long-time behavior away from the ground states} \label{Section:LongTimeFarGd}
In this section, we prove Proposition \ref{Prop:FarFromGd}.
Let $u$ be a solution of \eqref{Eqn:SchrodCrit} satisfying $u([t_0,T_+(u))\subset\cH^{\ep}\setminus\ti B_\de(\cW)$ for some $\ep\in(0,\ep_B(\de))$.
By Remark \ref{in ckcH}, $u$ stays in $\ck\cH$, so $\Te(u)\in\{\pm 1\}$ is a constant.
Moreover, Proposition \ref{Prop:staying} implies that $t_+(\de')<T_+(u)$ for all $\de'\in[\de,\de_B]$, so Proposition \ref{Prop:OnePassLemma} yields $t_1\in I(u)$ such that
\EQ{
 u([t_1,T_+(u))) \subset \cH^\ep \setminus \ti B_{\de_B}(\cW).}
Without losing generality, we may assume $t_1=0$ by time translation.

\subsection{Blow-up after ejection}
In the case of $\Te(u)=-1$ and $u_0 \in H^1\rad$, we prove that $T_{+}(u) < \infty$.
Let $m \gg 1$. We rewrite \eqref{Eqn:RHS} in the following way
\EQ{ \label{Eqn:OgawTsut}
 \dot \sV_m = 2 K(u) - 2\LR{|u_r |^{2},f_{0,m}}
  + \frac 12\LR{|u/m|^2,f_{1,m}}  + 2\LR{|u|^{6},f_{2,m}}}
with
\EQ{
 \pt f_{0} := 1 - \phi - r\p_r\phi,
 \pq f_{1} := - \De (r\p_r+3)\phi,
 \pq f_{2} := 1 - \phi - r\p_r\phi/3.}
By the property of $\phi$, we have $\supp f_{j,m}\subset\{m\le |x|\le 2m\}$ and  $0\le f_{2,m}\le f_{0,m}$.
Hence using \eqref{ext GN} and the $L^2$ conservation, we obtain
\EQ{ \label{Eqn:FubInteg2}
 \LR{|u|^6,f_{2,m}}  \pt\lec  \int_{m}^{\infty} f_{0,m}(r) |u|^{6} r^2 dr
 \pr= \int_{m}^{\infty} \int_{m}^{r} f_{0,m}' (s) ds |u|^6  r^2 dr
 \pn\sim \int_{m}^{\infty} f_{0,m}'(s) \| u \|^{6}_{L^{6}(|x|> s )}ds
 \pr\lec \int_{m}^{\infty} f_{0,m}'(s) \frac{1}{s^{4}} \| \nabla u \|^{2}_{L^{2}(|x|>s)} \| u \|_{L^{2}(|x|>s)}^4 ds
 \pr\lec \frac{\|u_0\|_{L^2}^4}{m^4} \int_m^\I f_{0,m}|u_r|^{2} dx
 \sim \frac{\|u_0\|_{L^2}^4}{m^4}\LR{|u_r|^2,f_{0,m}}.}
We also have $\int_{\R^3} |f_{1,m}(r)| |u/m|^{2} dx \lec m^{-2}\|u_0\|_{L^2}^2$. Hence for $m\gg \|u_0\|_{L^2}/\ka(\de_B)$ and $0<t<T_+(u)$ we have
\EQ{ \label{Vm bd}
 -\dot\sV_m(t) \ge -K(u(t)) \ge \ka(\de_B)>0.}
Now assume for contradiction that $u$ exists for all time $t>0$, namely $T_{+}(u)= \infty$.
Then choosing $m\gg \|u_0\|_{L^2}/\ka(\de_B)$, we have from \eqref{Vm bd}
\EQ{
 m \|u_r (t) \|_{L^{2}}\|u_0\|_{L^2} & \gtrsim - \sV_m(t) \to \infty,}
as $t \to \infty$, hence
\EQ{
 -K(u(t)) & = -6 E(u) + 2 \|u_r(t) \|^{2}_{L^{2}} \to  \infty.}
So one can choose $T_{1} > 0$ such that $ -K(u(t))  \sim  \|u_r (t) \|^{2}_{L^{2}}$ for $ t \geq T_{1}$. Hence
\EQ{
 m \|u_r(T_{2}) \|_{L^{2}} & \geq - m \|u_r(T_{1}) \|_{L^{2}} + \frac{c}{\|u_0\|_{L^2}} \int_{T_{1}}^{T_{2}} \|u_r(t) \|^{2}_{L^{2}}dt }
for $T_{2} \geq T_{1}$ and some absolute constant $c\in(0,1)$.
Therefore, defining
\EQ{
 f(t):=  - m \|u_r(T_1) \|_{L^{2}} + \frac{c}{\|u_0\|_{L^2}} \int_{T_{1}}^{t} \|u_r(s) \|^{2}_{L^{2}}ds,}
we see that for large $t>T_1$, $f(t)$ is positive and $\p_t f(t) \gtrsim f(t)^{2}$.
Integrating this differential inequality yields a singularity and therefore blow-up in finite time.

\subsection{Scattering after ejection}
In the case of $\Te(u)=+1$, the proof of scattering uses arguments from \cite{kenmer} with arguments from \cite{nakaschlagschrod}.
Unlike the subcritical case, we have to take account of the scaling parameter and the fact that the maximal time interval of existence might be finite, even though the $\dot{H}^{1}$ norm is bounded by \eqref{unif bd Te+}.

We recall the following result proved by Keraani \cite{keraani} using a concentration compactness procedure, cf.~\cite{bahger,lions,merlevega}.
Since we are dealing with radial solutions only, we restrict it to the radial case.
\begin{lem}[\cite{keraani}]
Let $\{ v_{0,n} \}_{n \geq 1} $ be a bounded sequence in $\dot{H}^{1}\rad$. Then, passing to a subsequence, there exist sequences $\{ V^{j} \}_{j \geq 0} \subset\dot{H}^{1}\rad$ and $\{(\si_{j,n}, t_{j,n})\}_{j\ge 0,n\ge 1}\subset \R^2$ with the following properties. For each $j \neq j'$
\EQ{ \label{Eqn:OrthCond}
 \lim_{n\to\I}|\si_{j,n}-\si_{j',n}| + |e^{-2\si_{j,n}}(t_{j,n} - t_{j',n})|=\I.}
For $\ga_{k,n}(t,x)$ defined by
\EQ{ \label{Eqn:VonDecomp}
 e^{- i t \De} v_{0,n} & = \sum_{j=0}^{k} e^{-i(t+t_{j,n}) \De}S_{-1}^{-\si_{j,n}} V_{j} + \gamma_{k,n},}
we have
\EQ{ \label{Eqn:GammaStrich}
 \lim_{k \to \infty} \limsup_{n \to \infty} \| \gamma_{k,n} \|_{S} & =0.}
For all $k$ and as $n\to\I$,
\EQ{ \label{Eqn:KinetNrj}
  \|v_{0,n} \|_{\dot{H}^{1}}^{2}
  =\sum_{j=0}^{k} \| V_{j} \|^{2}_{\dot{H}^{1}} + \| \gamma_{k,n}(0) \|^{2}_{\dot{H}^{1}} +o(1)}
and, putting $s_{j,n}:=t_{j,n}e^{-2\si_{j,n}}$,
\EQ{ \label{Eqn:SeparNrj}
 E(v_{0,n}) = \sum_{j=0}^{k} E( e^{ - i s_{j,n}\De } V_{j}) + E(\gamma_{k,n}(0)) +o(1).}
\end{lem}
For any $A<E(W)+\ep_S^2$ and any $\de\in(0,\de_B]$, let $\sS(A,\de)$ be the collection of solutions of \eqref{Eqn:SchrodCrit} such that
\EQ{
 E(u)\le A, \pq u([0,T_+(u))) \subset \ck\cH\setminus\ti B_{\de}(\cW),
 \pq \Te(u(0))=+1.}
Since $u$ stays in $\ck\cH$ for $0\le t<T_+(u)$, $\Te(u(t))=+1$ is preserved.

It is well known that $u\in S(0,\I)$ implies the scattering as $t\to\I$, see \cite{cazbook} or \cite[Remark 2.15]{kenmer}.
Define the minimal energy where uniform Strichartz bound fails.
\EQ{
 \pt S(A,\de):=\sup_{u\in\sS(A,\de)}\|u\|_{S(0,T_+(u))},
 \pr E_c(\de):=\sup\{A<E(W)+\ep_S^2 \mid S(A,\de)<\I\}.}
Notice that $S(A,\de)\lec A^{1/2}$ holds for $0<A \ll 1$, by the small data scattering (see \cite{cazbook} for example).
Moreover, the result of \cite{kenmer} implies $E_c(\de)\ge E(W)$.
If $E_{c}(\de)<E(W)+\ep_S^2$, there exists a sequence of solutions $u_{n}\in\sS(A_n,\de)$ for some sequence of numbers $A_n\to E_c(\de)$ such that
\EQ{ \label{Sbup un}
 \| u_{n} \|_{S( 0 , T_{+}(u_{n}))} \to \I.}

Next we prove the existence of a critical element:
\begin{lem} \label{Claim:Crit}
Let $\delta \in (0,\delta_{B})$. Suppose that $E_c(\de)\le E(W)+\ep^2$ for some $\ep$ such that 
\EQ{
 0< \ep < \min(\ep_V(\de),\ep_B(\de_B)), \pq \ep \ll \min(\ep_S,\de,\sqrt{\kappa(\de)}).}
Let $A_n\to E_c(\de)$ and $u_n\in\sS(A_n,\de)$ satisfying \eqref{Sbup un}.
Then there exist $U_{c}\in\sS(E_c(\de),\de_B)$ satisfying $E(U_{c})= E_{c}(\de)$ and $ \| U_{c} \|_{S(0, T_{+}(U_{c}))} = \infty$, and $(\si_n,s_n)\in\R^2$ such that $e^{is_n\De}S_{-1}^{\si_n}u_n(0)$ is strongly convergent in $\dot H^1\rad$.
\end{lem}
Note that once we have $U_c\in \sS(E_c(\de),\de)$ with the other properties, then a time translation yields another minimal element in $\sS(E_c(\de),\de_B)$ as a consequence of the ejection and the one-pass lemmas (see the proof below for the detail). 
\begin{proof}
Notice that, by the small data scattering, cf.~\cite[Remark 2.7]{kenmer}, we must have $ \|u_{n} \|_{\dot H^1} \gtrsim 1 $, otherwise $\|u_n\|_{S(0,\I)}$ are uniformly small.
Hence using \eqref{unif bd Te+} as well, we have for all $n$ and $t\in[0,T_+(u_n))$,
\EQ{
 \|u_n(t)\|_{\dot H^1}\sim 1.}

We then apply \eqref{Eqn:VonDecomp} to $v_{0,n}:= u_{n}(0)$. Then we have, up to a subsequence,
\EQ{ \label{Eqn:ProfDecompun}
 e^{- i t \De} u_{n}(0) & = \sum_{j=0}^{k} e^{- i t \De} S_{-1}^{-\si_{j,n}}e^{- i  s_{j,n}\De } V_{j} + \gamma_{k,n}.}
Let $s_{j,\infty} \in [- \infty, \infty]$ such that $s_{j,n} \to s_{j,\infty}$ (up to a subsequence), and let $U_{j}$ be the nonlinear profile associated with $(V_{j}, \{ s_{j,m} \}_{m \geq 1})$, that is the unique solution of \eqref{Eqn:SchrodCrit} around $t=s_{j,\infty}$ satisfying (see \cite{kenmer} for more detail),
\EQ{ \label{Eqn:DfnNonlinearProf}
 \lim_{ m \to \infty }  \| U_{j}(s_{j,m}) -  e^{-i s_{j,m} \De} V_{j} \|_{\dot{H}^{1}} & =0.}
We also define $U_{j,n}(t):= S_{-1}^{-\si_{j,n}}  U_{j}( (t + t_{j,n}) e^{-2\si_{j,n}} )$.
Since $\ep<\ep_V(\de)$, Proposition \ref{Prop:VarEst} implies that $K(u_n(t)) \gec \kappa(\de)\gg \ep^2$ for all $n\ge 1$ and $t\ge 0$.
Then using \eqref{Eqn:KinetNrj} and the conservation of $G$ for the free equation, we have
\EQ{ \label{Eqn:DecompEG}
 E(W) - \epsilon^{2} & > E(u_{n}) - K(u_{n}(0))/6 + \epsilon^{2} \\
& = G(u_{n}(0)) + \epsilon^{2}
 \geq \sum_{j=0}^{k} G(V_{j}) + G(\gamma_{k,n}).}
This implies that $G(e^{-it \De} V_{j}) \leq E(W)- \epsilon^{2}$ and
$G(\gamma_{j_{k,n}}) \leq E(W) - \epsilon^{2}$.
By \eqref{Eqn:Charac3GdState}, $K(\gamma_{k,n}(t))\ge 0$ and so $E(\gamma_{k,n}(t))\ge 0$ for all $t\in\R$.
Similarly, $ K(e^{-it \De} V_{j})\ge 0 $ for all $t\in\R$, and $K(e^{-it\De}V_j)\to\|V_j\|_{\dot H^1}^2$ as $t\to\pm\I$, and both can be zero only if $V_j=0$.
Hence by \eqref{Eqn:DfnNonlinearProf}, $K(U_{j}(t)) \ge 0$ in a neighborhood of $s_{j,\infty}$, and $ E(U_{j}) \ge 0$.

By \eqref{Eqn:SeparNrj} and \eqref{Eqn:DfnNonlinearProf}, we have as $n\to\I$,
\EQ{ \label{Eqn:DecompNrj}
 E(u_{n}) & =\sum_{j=0}^{k} E(U_{j}) + E(\gamma_{k,n}(0)) + o(1),}
hence we see that for all $j\ge 0$
\EQ{ \label{Eqn:BoundEUj}
 E(U_{j}) & \leq E_{c}. }
If $E(U_{j}) < E(W)$, then we conclude from $K(U_{j})\ge 0$ in a neighborhood of $s_{j,\infty}$ and \cite{kenmer} that $U_{j}$ exists globally in time and scatters with $\| U_{j} \|_{S\cap W^1\cap L^\I_t\dot H^1} < \infty$.

Assuming that $\|U_j\|_{S}<\I$ for all $j=0,1,\dots,k$,
we apply Lemma \ref{res:perturbation} to
\EQ{
 \tilde u:= \sum_{j=0}^{k} U_{j,n}+\ga_{k,n}, \pq u:=u_{n}}
from $t_{0}:=0$ on $I:=\R$.
\eqref{Eqn:RelNrjK} and \eqref{Eqn:DecompNrj} imply that $\ti u$ is bounded in $L^\I_t \dot H^1$ as $n\to\I$ uniformly in $k$.
From \eqref{Eqn:KinetNrj}, the orthogonality conditions \eqref{Eqn:OrthCond} and a similar argument to that in the proof of Proposition 4.2 in \cite{kenmer}, $\tilde{u}$ is bounded in $S$ as $n \rightarrow \infty$ uniformly in $k$.
\eqref{Eqn:DecompEG} implies that $\|\ti u(0)-u_n(0)\|_{\dot H^1}\to 0$ as $n\to\I$.
Hence, in order to apply the lemma for large $n$, it suffices to make
\EQ{
 i\p_t\ti u-\De\ti u-|\ti u|^4\ti u=\sum_{j=0}^k |U_{j,n}|^4U_{j,n}-|\ti u|^4\ti u}
small in $N^1(\R)$.
Indeed, using $\| U_{j} \|_{S\cap W^1} < \infty$ and the orthogonality conditions \eqref{Eqn:OrthCond}, as well as \eqref{Eqn:GammaStrich}, we obtain
\EQ{
  \lim_{k\to\I}\limsup_{n\to\I}\|i\p_t\ti u-\De\ti u-|\ti u|^4\ti u\|_{N^1(\R)}=0.}
For a proof, we refer again to Proposition 4.2 in \cite{kenmer} and the references therein (in particular \cite{keraani}).
Thus for large $k$ and large $n$, Lemma \ref{res:perturbation} yields a bound on $\|u_n\|_{S(\R)}$ uniform in $n$, contradicting $\|u_n\|_{S(0,T_+(u_n))}\to\I$.

This means that there exists at least one $U_{j}$ such that $E(W)\le E(U_{j})\le E_c(\de)$ and $\|U_j\|_{S(I(U_j))}=\I$, say $U_{0}$.
Then, by the orthogonality \eqref{Eqn:DecompNrj} and the positivity of $K(U_{j})\ge 0$, we see that $T_{+}(U_j) = T_{-}(U_j) = \infty$ for
$j=1,..,k$ and
\EQ{ \label{Eqn:UnifBound}
 \sum_{j=1}^{k} \| U_{j} \|^{2}_{L_{t}^{\infty} \dot{H}^{1}} + \| \gamma_{k,n} \|^{2}_{\dot{H}^{1}} & \lesssim \epsilon^{2}\ll 1,}
hence all the $U_j$, $j=1,\dots,k$, are small and scatter.
That allows us to apply Lemma \ref{res:perturbation} to $\ti u$ and $u=u_n$ on $I=I_n:=e^{2\si_{0,n}}I_0 - t_{0,n}$ for any interval $I_0\subset I(U_0)$ such that $\|U_0\|_{S(I_0)}<\I$. Then the lemma yields
\EQ{
  \limsup_{n\to\I}\|u_n\|_{S(I_n)}<\I, \pq
  \lim_{k\to\I}\limsup_{n\to\I}\|\ti u-u_n\|_{L^\I_t \dot H^1(I_n)}=0,}
since $B_3+\nu\to 0$.
Combining the second estimate with \eqref{Eqn:UnifBound} and the orthogonality \eqref{Eqn:OrthCond}, we see that
\EQ{ \label{app un by U0}
 \limsup_{n\to\I}\|u_n-U_{0,n}\|_{L^\I_t \dot H^1(I_n)} \lec \ep.}

Suppose that $s_{0,\infty} = +\infty$.
Then by definition $U_{0}$ scatters in a neighborhood of $\infty$, namely $\|U_0 \|_{S( [T_{0}, \infty))} < \infty$ for some $T_0\in I(U_0)$.
Hence, choosing $I_0=[T_0,\I)$ in the above argument yields a uniform bound on $\|u_n\|_{S(I_n)}$, but $I_n\supset [0,\I)$ for large $n$, contradicting $\|u_n\|_{S(0,T_+(u_n))}\to\I$.

Suppose that $s_{0,\infty} = - \infty$.
Then by definition $\| U_0 \|_{S(-\I,t_0)} <\infty$ for some $t_0\in I(U_0)$.
Then $U_0\not\in S(I(U_0))$ implies $\|U_0\|_{S(t_0,T_+(U_0))}=\I$.

Suppose that $s_{0,\infty} \in \mathbb{R}$.
If $\| U_{0} \|_{S (s_{0, \I},  T_{+}(U_{0}) ) } < \infty$,
then the blow-up criterion (see for example \cite{kenmer}) implies $T_{+}(U_{0}) = \infty$, and so choosing $I_0=(s_-,\I)$ for some $s_-\in (-T_-(U_0),s_{0,\I})$, we see that $I_n\supset[0,\I)$ for large $n$, leading to a contradiction with $\|u_n\|_{S(0,\I)}\to\I$ as in the case $s_{0,\I}=+\I$ above.
Therefore $\|U_0\|_{S(s_{0,\I},T_+(U_0))}=\I$.

Thus we have obtained $s_{0,\I}<+\I$ and $\|U_0\|_{S(t_0,T_+(U_0))}=\I$ for some $t_0\in(s_{0,\I},T_+(U_0))$.
Since $E(U_0)\le E_c(\de)\le E(W)+\ep^2$ and $\ep<\ep_B(\de_B)$, Propositions \ref{Prop:OnePassLemma} and \ref{Prop:staying} imply that there are only two options for $U_0$:
\begin{enumerate}
\item There exists $t_+\in I(U_0)$ such that $\ti d_\cW(U_0(t))\ge \de_B$ for all $t_+<t<T_+(U_0)$.
\item $\limsup\limits_{t\nearrow T_+(U_0)}\ti d_\cW(U_0(t))\le \ep/c_D$.
\end{enumerate}
In the second case, choosing $I_0=(t_0,t_c)\subset I(U_0)$ such that $\ti d_\cW(U_0(t_c))<2\ep/c_D$, we obtain from Proposition \ref{Prop:DistFuncEigen}
and \eqref{app un by U0} that
\EQ{
 \ti d_\cW(u_n(t_n)) \lec \ti d_\cW(U_0(t_c)) + \|u_n(t_n)-U_{0,n}(t_n)\|_{\dot H^1}\lec \ep \ll \de,}
where $t_n:=e^{2 \sigma_{0,n}}t_c-t_{0,n}>0$ for large $n$, since $t_c>s_{0,\I}$.
This contradicts $\ti d_\cW(u_n(t))\ge \de$ on $[0,T_+(u_n))$.
Therefore $\ti d_\cW(U_0(t))\ge\de_B$ for $t_{+} \le t<T_+(U_0)$.

Since $E(U_0)\le E_c(\de)\le E(W)+\ep^2$ and $\ep<\ep_B(\de_B)$, $\Te(U_0(t))$ is constant on $[t_+,T_+(U_0))$.
Let $t_n:=e^{2 \sigma_{0,n}}t_+-t_{0,n}$, then $\Te(U_{0,n}(t_n))=\Te(U_0(t_+))$ by the invariance of $\Te$.
\eqref{app un by U0} implies for large $n$ that $u_n(t_n)$ is in an $O(\ep)$ ball around $U_{0,n}(t_n)$, which is included in $\cH^{\ep_S}\cap\ck\cH$ because $\ep\ll\min(\ep_S, c_D \de)$.
Hence
\EQ{
 +1=\Te(u_n(t_n))=\Te(U_{0,n}(t_n))=\Te(U_0(t_+)).}
Therefore, putting $U_c(t):=U_0(t+\tilde{t})$ with $\tilde{t}:= \max{(t_0,t_+)}$, we obtain $U_c\in\sS(E_c(\de),\de_B)$ and $\|U_c\|_{S(0,T_+(U_c))}=\I$.
Then the definition of $E_c$ implies $E(U_c)\ge E_c(\de)$, and so $E(U_c)=E_c(\de)$.
Thus $U_c$ satisfies all the properties in the lemma.

Since $E(u_n)\to E_c=E(U_0)$ and $K(U_j)\ge 0$, by \eqref{Eqn:RelNrjK} and \eqref{Eqn:DecompNrj} we have $U_{j}=0$ for all $j \geq 1$ and $\gamma_{k,n} \to 0$ in $\dot{H}^{1}$. Hence
$u_{n} (0) = S_{-1}^{-\si_{0,n}} e^{- i s_{0,n} \De } V_0 + o(1)$ in $\dot H^1$. Hence $e^{i s_n \triangle}S_{-1}^{\sigma_n} u_n(0) \rightarrow V_0$
with $\sigma_n:= \sigma_{0,n}$ and $s_n := s_{0,n}$.
\end{proof}

The following is a corollary of the last part.
\begin{claim} \label{Claim:Precompact}
There exists $\si_c:[0, T_{+}(U_{c}))\to\R$ such that
\EQ{ \label{Eqn:TrajUcComp}
 \sK & := \{S_{-1}^{-\si_c(t)} U_c(t)\}_{0\le t< T_{+}(U_{c})}\subset\dot H^1\rad}
is precompact.
\end{claim}
\begin{proof}
If there is no such $\si_c$, then there exists $\{ t_{n} \}_{n \geq 1} $ and $ \eta > 0$   such that
\EQ{ \label{Eqn:IfNotPrec}
 \inf_{\si\in\R} \|S_{-1}^\si U_{c}(t_{n}) -U_{c}(t_{n'})\|_{\dot{H}^{1}} & \geq \eta}
for all $n \neq n'$.
Notice that we must have, after possibly passing to a subsequence, $ t_{n} \to T_{+}(U_{c})$: otherwise, we get a contradiction from \eqref{Eqn:IfNotPrec} with $\si_0=0$ by continuity of $U_c(t)$.
Applying Lemma \ref{Claim:Crit} to $u_{n}(t):= U_{c}(t+t_{n})$ yields a sequence $(\si_n,s_n)\in\R^2$ such that
$e^{is_n\De}S_{-1}^{\si_n}U_c(t_n)$ is strongly convergent in $\dot H^1$.

After possibly passing to a subsequence, we may assume that $s_n$ converges to some $s_\I \in [-\infty, \infty ]$.
If $s_\I\in \mathbb{R}$, then $S_{-1}^{\si_n}U_c(t_n)$ is also convergent, contradicting \eqref{Eqn:IfNotPrec} for $n,n'\to\I$.
If $s_\I = - \infty$, then $\| U_{c} (t+ t_{n}) \|_{S(0, \infty)}\to 0 $, and if $s_\I = \infty$, then $\| U_{c}(t+ t_{n}) \|_{S(- \infty, 0)}\to 0$, since the free solutions with the same data as $U_c$ at $t=t_n$ are vanishing in that way: see \cite{kenmer} for more detail.
In either case, it contradicts $\| U_{c} \|_{S(0, T_{+}(U_{c}) )} = \infty$.
\end{proof}

We are now ready for the final step of the proof of Proposition \ref{Prop:FarFromGd}.
\begin{claim} \label{Claim:Notexist}
$U_{c}$ does not exist.
\end{claim}
\begin{proof}
First we consider the case $T_{+}(U_{c}) < \infty$.
The local wellposedness theory, together with the precompactness of $\sK$, implies that blow-up is possible only by concentration $\si_c(t) \to \infty$ as $t\nearrow T_+(U_c)<\I$, see \cite{kenmer} for a proof.
For any $m>0$ and $t\in I(U_c)$, put
\EQ{
 y_m(t):= \LR{|U_{c}(t)|^{2},\phi_{m}}.}
Then as $t\nearrow T_+(U_c)$, we have
\EQ{
 y_m=\LR{|S^{-\si_c}_{-1}U_c|^2,e^{-2\si_c}\phi_{me^{\si_c}}} \to 0, }
because $|S^{-\si_c}_{-1}U_c|^2$ is precompact in $L^3_x$ while $e^{-2\si_c}\phi_{me^{\si_c}}\to 0$ weakly in $L^{3/2}_x$.
Using \eqref{Eqn:IdentLocMass} and Hardy's inequality, we have
\EQ{
 |\dot y_m(t)| \le 2|\LR{\na U_{c},U_{c}\na\phi_m}| \lec  \|U_c\|_{\dot H^1}^2 \lec 1,}
uniformly in $m>0$.
Integrating it on $t<T_+(U_c)$ and sending $m\to\I$, we obtain $\|U_c(t)\|_{L^2}^2\lec |T_+(U_c)-t|$ and so $U_c(0)\in L^2_x$. Hence by the $L^2$ conservation, we get
$\| U_{c}(0) \|_{L^{2}} = \| U_{c}(t) \|_{L^{2}} \rightarrow 0$ as $t\nearrow T_+(U_c)$.
So $U_c =0$ and it contradicts $T_{+}(U_c) < \infty$. 

Therefore $T_+(U_c)=\I$. For all $t\in[0,\I)$, we have $\ti d_\cW(U_c(t))\ge \de_B$, and also $\|U_c(t)\|_{\dot H^1}\gec 1$ by the small data scattering.
Hence Proposition \ref{Prop:VarEst} implies that
\EQ{
 \ti\ka:=\inf_{t\ge 0}K(U_c(t))>0.}

Suppose that
\EQ{
 A:=\inf_{0\le t<\I} \si_c(t) >-\I.}
Then by precompactness of $\sK$ and Hardy's and Sobolev's inequalities, there exists $m$ such that
\EQ{ \label{Eqn:Vdecay}
 \int_{|x|> m} | \nabla U_c|^{2} + |U_c|^6 + |U_c/r|^2 dx  \ll \ti\ka}
for all $t \in [0, \infty)$, while $\sV_m(t)$ is bounded for $t\to\I$.
Applying \eqref{Eqn:OgawTsut} to $U_{c}$, integrating it on $[0,T]$ with $T\to\I$, we get a contradiction from $T\ti\ka \le [\sV_m]_{0}^{T}$.

Therefore $A=-\I$. Then by continuity of $U_c(t)$, we deduce that $\si_c(t_n)\to-\I$ along some sequence $t_n\to\I$ satisfying
\EQ{ \label{choice tn}
 \min_{0\le s\le t_n}\si_c(s)=\si_c(t_n).}
By the precompactness of $\sK$, we may assume that $S^{-\si_c(t_n)}_{-1}U_c(t_n)$ converges strongly in $\dot H^1\rad$.
Let $U_n$ and $U_\om$ be the solutions of \eqref{Eqn:SchrodCrit} with the initial data
\EQ{
 U_n(0)=S^{-\si_c(t_n)}_{-1}U_c(t_n), \pq U_\om(0)=\lim_{n\to\I}U_n(0).}
The local wellposedness theory implies that for any compact $J\subset I(U_\om)$, $U_n\to U_\om$ as $n\to\I$ in $C(J;\dot H^1_x)\cap S(J)$.
This convergence in $S$ and
\EQ{ \label{Un S bup}
 \|U_n\|_{S(-t_ne^{2\si_c(t_n)},0)}=\|U_c\|_{S(0,t_n)} \to \|U_c\|_{S(0,\I)}=\I}
imply that for each $t\in(-T_-(U_\om),0]$ and large $n$, we have $|t|<t_ne^{2\si_c(t_n)}$. Then putting $s_n:=t_n-|t|e^{-2\si_c(t_n)}\in(0,t_n]$, we have by the scale invariance,
\EQ{
 S_{-1}^{\si_c(t_n)-\si_c(s_n)}U_n(t)=S_{-1}^{-\si_c(s_n)}U_c(s_n)\in\sK\setminus\ti B_{\de_B}(\cW).}
Since $U_n(t)\to U_\om(t)$ in $\dot H^1$ and $\sK$ is precompact,
$\si_c(s_n)-\si_c(t_n)$ converges to some $\si_\om(t)\in[0,\I)$ up to a subsequence, where positivity comes from \eqref{choice tn}.
Then
$\{S_{-1}^{-\si_\om(t)}U_\om(t)\}_{t\in(-T_-(U_\om),0]}$
is in the closure of $\sK$, hence precompact, and also, $\ti d_\cW(U_\om(t))\ge\de_B$ for all $t\in(-T_-(U_\om),0]$.
Moreover $\|U_\om\|_{S(-T_-(U_\om),0)}=\I$, since otherwise the blow-up criterion yields $T_{-}(U_{\om}) = \infty$ and
the long-time perturbation for $t<0$ yields a uniform bound on $\|U_n\|_{S(-\I,0)}$ for large $n$, contradicting \eqref{Un S bup}.

Thus we have obtained another critical element $\bar U_\om(-t)$, that is the time inversion of $U_\om$, with the scale bound $\si_\om\ge 0$. Hence the above argument for $A>-\I$ applied to this new critical element yields a contradiction.
\end{proof}

\section{Four sets of dynamics} \label{Sec:ProofThm2}
In this section, we prove Theorem \ref{Thm:4 sets}.
Let  $ 0 < \beta \ll \epsilon_\star$ and $R > 0$ be such that $\|\phi_R^C W\|_{\dot H^1}^2 \le \be^4$.
We consider four solutions $u$ around $\cW$ with the following initial data at $t=0$ in the coordinate \eqref{coord around W} with $\vec\la:=(\la_1,\la_2)$,
\EQ{
 \pt \gamma(0) := - \phi_R^C W + \omega(\phi_R^C W , g_{-}) g_{+} -\omega(\phi_R^C W , g_{+}) g_{-},
 \pr \vec\la(0) = \beta (\pm 1, 0), \beta (0, \pm 1).}
Note that $\ga(0)$ is the symplectic projection of $-\phi_R^CW$ to the subspace that is perpendicular (with respect to $\omega$) to $\Span\{g_{-},g_{+}\}$.
This ensures $u(0)\in L^2_x$ so that we can apply the above blow-up result.

Let $I_E(u)\subset I(u)$ be the maximal interval where $u(t)\in B_{\de_E}(\cW)$ so that we can use the coordinate \eqref{coord around W}.
For brevity, put $\ti d(t):=\ti d_\cW(u(t))$ on $I_E(u)$.
Then by Proposition \ref{Prop:DistFuncEigen} and by the same argument as in \eqref{Eqn:EstDerEuEgamma}--\eqref{Eqn:EstPointWiseNrj}, we get
\EQ{
 \pt\| \gamma(t) \|^{2}_{\dot{H}^{1}} + O(\ti d(t)^4)
 \lec \be^4 + \int_{0}^{\ta(t)} (\ti d^2 \| \gamma \|_{\dot{H}^{1}_x} + \ti d^4) d\ta}
within $I_E(u)$. Hence on any interval $J\subset I_E(u)$ where $|\ta|\le \de_E^{-1/2}$, we have
\EQ{
 \| \gamma \|_{L_t^{\infty}(J;\dot{H}^{1})}^2
  \lec (1+\de_E^{-1})\|\vec\la\|_{L^\I_t}^4 \lec \|\vec\la\|_{L^\I_t(J)}^3.}
Then from \eqref{Eqn:EstDynLambda} and a continuity argument we see that
\EQ{ \label{hyperbolic}
\CAS{ \vec\la(0)= \beta ( \pm 1,  0)  \implies  \vec\la = \pm\beta  \left(\cosh(\mu \tau),\sinh (\mu \tau) \right)(1+O(\be^{1/2})) \\
 \vec\la(0)= \beta (0, \pm 1)  \implies \vec\la = \pm\beta\left(\sinh(\mu \tau),\cosh(\mu \tau) \right)(1+O(\be^{1/2})) } }
as long as
\EQ{ \label{hyper valid}
 |\ta|\le\de_E^{-1/2},
 \pq \be e^{\mu|\ta|} \sim \ti d(t)<\de_E.}
Using \eqref{Eqn:ExpEnergy} and \eqref{Eqn:ExpandK2nd}, we see that if $\vec\la(0)=\pm\be(1,0)$, then $E(u)-E(W)\sim-\be^2<0$ and $K(u(0))\sim \mp \be$.
Hence by \cite{kenmer},
\EQ{
 \CAS{\vec\la(0)=\be(1,0) \implies u(0)\in\cB_-\cap\cB_+,\\
 \vec\la(0)=-\be(1,0) \implies u(0)\in\cS_-\cap\cS_+.}}

If $\vec\la(0) := \pm\be(0,1)$, then
$0<E(u)-E(W)\sim\be^2\sim\ti d(0)^2\ll\ep_\star^2$, while near the boundary of the interval \eqref{hyper valid}, we have
$\ti d(t) \sim \min(\de_E,\be e^{\mu\de_E^{-1/2}}) \gg \be$.
Therefore Proposition \ref{Prop:DynEjection} applies to $u$ at some $t_+>0$ in the forward direction and at some $t_-<0$ in the backward direction, both within the interval \eqref{hyper valid}, where we have \eqref{hyperbolic}, and also
\EQ{
 \|\ga(t_\pm)\|_{\dot H^1} \ll |\la_1(t_\pm)| \sim |\la_2(t_\pm)|.}
Hence by \eqref{Eqn:ExpandK2nd}, we have $\sign K(u(t_\pm))=-\sign \la_1(t_\pm)=\Te(u(t_\pm))$ and
Proposition \ref{Prop:FarFromGd} yields
\EQ{
 \CAS{\vec\la(0)=\be(0,1) \implies \Te(u(t_\pm))=\mp 1,\pq \implies u(0)\in\cS_-\cap\cB_+,\\
 \vec\la(0)=-\be(0,1)\implies \Te(u(t_\pm))=\pm 1,\pq \implies u(0)\in\cS_+\cap\cB_-.}}

It is obvious that the above argument is stable for adding small perturbation in $\R^2$ to $\vec\la(0)$ and small perturbation in $H^1$ (in the orthogonal subspace) to $\ga(0)$.
Hence we obtain a small open set in $H^1$ around each of the four solutions.
\qedsymbol

\end{document}